\documentclass{article}
\usepackage[a4paper]{geometry}
\usepackage[utf8]{inputenc} 
\usepackage[english]{babel}
\usepackage{epsfig}
\usepackage{amsmath, amsfonts, amstext, amscd, bezier,dsfont,bm}
\usepackage{amssymb,amsthm}
\usepackage{indentfirst}
\usepackage{graphicx}
\usepackage{float}
\usepackage{subcaption}
\usepackage{color}
\usepackage{hyperref}
\usepackage[affil-it]{authblk}
\usepackage[ruled,vlined]{algorithm2e}
\usepackage{booktabs}
\usepackage{multirow}
\usepackage{xfrac}

\usepackage{xcolor}

\definecolor{DarkGreen}{RGB}{0,90,0}

\usepackage{mathtools}
\mathtoolsset{showonlyrefs}

\usepackage[authoryear]{natbib}

\newcommand{\A}{\mathbf{A}_{n\times n}}
\renewcommand{\a}{\mathbf{a}_{n\times n}}
\newcommand{\x}{\mathbf{x}_{n\times n}}

\newcommand{\Z}{\mathbf{Z}_n}
\newcommand{\z}{\mathbf{z}_n}
\newcommand{\e}{\mathbf{e}_n}
\newcommand{\hz}{\hat{\mathbf{z}}_n}
\renewcommand{\P }{\mathbb{P}_{P}}
\newcommand{\Pp }{\mathbb{P}_{\pi}}
\newcommand{\Pt }{\mathbb{P}_{\pi,P}}
\newcommand{\Pto }{\mathbb{P}_{\pi_0,P_0}}
\newcommand{\Po}{\mathbb{P}_{P_0}}
\newcommand{\pen}{\text{pen}}

\newcommand{\ind}{\mathds 1}
\newcommand{\Q}[1]{\mathrm{Q}_k(#1)}
\newcommand{\hk}{\hat{k}_{\text{\tiny DNML}}}
\newcommand{\hkNML}{\hat{k}_{\text{\tiny NML}}}
\newcommand{\hkNMCL}{\hat{k}_{\text{\tiny NMCL}}}
\newcommand{\pNML}{\text{pen}_{\text{\tiny NML}}}
\newcommand{\pDNML}{\text{pen}_{\text{\tiny DNML}}}
\newcommand{\NML}{\text{NML}}
\newcommand{\DNML}{\text{DNML}}
\newcommand{\CNML}{\mathcal C_{\text{\tiny NML}}}
\newcommand{\CDMNL}{\mathcal C_{\text{\tiny DNML}}^A}
\newcommand{\CDMNLz}{\mathcal C_{\text{\tiny DNML}}^z}
\newcommand{\NMCL}{\text{NMCL}}
\newcommand{\CNMCL}{\mathcal C_{\text{\tiny NMCL}}}

\theoremstyle{plain}
\newtheorem{theorem}{Theorem}
\newtheorem{lemma}[theorem]{Lemma}

\newtheorem{proposition}[theorem]{Proposition}

\theoremstyle{definition}

\theoremstyle{remark}

\numberwithin{equation}{section}
\numberwithin{theorem}{section}

\title{Normalized Likelihood Criteria for Model Selection in the Stochastic Block Model}

\author[1]{Andressa Cerqueira}
\author[1]{Felipe Baptistão}

\affil[1]{Departament of Statistics, Federal University of São Carlos}

\date{\today}

\begin{document}

\maketitle

\begin{abstract}
Estimating the number of communities is a fundamental problem in network analysis under the stochastic block model (SBM). In this paper, we study penalized estimators for this task based on normalized likelihood criteria. We show that a penalized estimator derived from the Normalized Maximum Likelihood (NML) is strongly consistent with a logarithmic penalty term, although its computation is intractable. To overcome this limitation, we consider the Normalized Maximum Complete Likelihood (NMCL) and the Decomposed Normalized Maximum Likelihood (DNML). The DNML admits an explicit formulation with cubic computational complexity in the number of nodes. We prove that the NMCL and DNML--based estimators are strongly consistent for sparse networks in which the average node degree diverges with the network size. Empirical results show that the DNML estimator performs competitively with existing methods, particularly in unbalanced networks.
\end{abstract}

\section{Introduction}

Network analysis has become a central area of research in understanding complex systems, with community detection being one of the most extensively studied problems. The Stochastic Block Model (SBM) proposed by \cite{holland1983stochastic} is a generative model for networks with community structure, in which nodes are partitioned into latent groups and edge probabilities depend on group memberships. This model has served as a foundation for both methodological developments and theoretical investigations. For theoretical results on the fundamental limits of community detection in the SBM, see \cite{Abbe2018} and references therein. Many community detection methods assume that the number of communities is known in advance, which makes the inferential problem of estimating this quantity a problem of independent interest.

Several methods based on different principles have been proposed in the literature to identify the number of communities in stochastic block models. Among them, network cross-validation methods have been proposed to select the number of communities. In \cite{Chen2028_NCV} the authors treat the rows of the adjacency matrix as data points to be considered in the cross validation algorithm while in \cite{Tianxi2020_ECV} the edges are considered as data points. In \cite{chakrabarty2025network}, the authors proposed a network cross-validation method based on overlapping partitions obtained through subsampling a set of overlapping nodes.

Methods based on sequential statistical testing have also been proposed in the literature to identify the number of communities. In \cite{Zhao2011_extraction}, the authors introduced a method that extracts one community at a time from the network, along with a sequential hypothesis testing procedure in which the null hypothesis states that the remaining network contains only one community. In \cite{Bickel2015_test}, the authors proposed a recursive bipartitioning algorithm based on a hypothesis test, where the null hypothesis is that the network has only one community, and they derived the corresponding null distribution. \cite{Lei2016_test} extended this approach by proposing a hypothesis test to test whether a network contains a specified number of communities $K$. In this work, the authors derived the null distribution under an SBM with more than one community and estimated the number of communities using the sequential goodness-of-fit test. Recently, \cite{oliveira2025counting} introduced a sequential statistical test for estimating the number of communities in weighted networks, in which the test statistic is derived from semidefinite programming. To estimate the number of communities using spectral methods, \cite{Can2022_spectral} proposed analyzing the spectral properties of the Bethe Hessian and non-backtracking matrices.

Another approach to model selection relies on likelihood-based methods through the derivation of BIC-type criteria, traditionally used in statistical inference. In particular, \cite{wang2017likelihood} proposed a penalized estimator for selecting the number of communities based on the maximum likelihood and established its consistency, showing that the estimator converges in probability to the true number of communities as the network size $n$ grows. The penalty function used in this estimator is of order $n\log n$. Although direct optimization of the likelihood function, which involves summing over all possible community assignments, is computationally intractable, this problem is commonly addressed using variational Expectation-Maximization (VEM) algorithms, which efficiently approximate the intractable marginal likelihood \citep{daudin2008mixture}.

Subsequently, \cite{hu2020corrected} considered a related approach based on the profile likelihood, where the latent community assignments are explicitly taken into account in the joint likelihood. They proved the consistency of the resulting penalized estimator and showed that, compared to the maximum likelihood–based approach of \cite{wang2017likelihood}, the associated penalty has a reduced order, consisting of a logarithmic term $\log n$ together with an additional term that is linear in the sample size. In \cite{cerqueira2020}, the authors replaced the likelihood and the profile likelihood with the integrated likelihood, obtained by integrating the likelihood function over the model parameters. They proved that the resulting estimator is strongly consistent, meaning that it converges almost surely to the true number of communities, for a penalty function of smaller order than in previous works, namely $n\log n$. Extensions of this estimator have been considered for model selection in the Degree-Corrected SBM \citep{Cerqueira24_DCSBM} and in multilayer and dynamic SBMs \citep{Arts2025}. Several other estimators based on the likelihood have been proposed and examined empirically, including those by \cite{daudin2008mixture}, \cite{latouche2012variational} and \cite{saldana2017many}; however, theoretical guarantees, including consistency, have not yet been derived for these approaches.

In this work, we propose a theoretical study of penalized estimators based on different forms of normalized likelihood. First, we prove that a penalized estimator based on the Normalized Maximum Likelihood (NML) is strongly consistent with a penalty term of order $\log n$. Although this estimator enjoys strong theoretical guarantees with a penalty term of order comparable to the minimal order proved in \cite{cerqueira2020}, its computation is computationally intractable, since its normalizing term requires summation over all possible networks, and the likelihood itself involves summation over all community configurations. To avoid computing the likelihood, we propose using the Normalized Maximum Complete Likelihood (NMCL). However, the computation of the normalizing constant also remains intractable in this case, as it requires summation over all possible networks. 

To address this computational issue, we also propose a penalized estimator based on the Decomposed Normalized Maximum Likelihood (DNML), which was originally introduced for general latent variable models by \cite{yamanishi2019decomposed}. Based on the seminal work of \cite{yamanishi2019decomposed}, it is possible to derive an explicit formula to compute the DNML, whose computational cost scales cubically with the network size. Unlike the NML and NMCL, which normalize over the entire network space jointly, the DNML decomposes the normalization across the network and the community assignments. Both the DNML and NMCL depend on estimated community assignments, which in practice can be obtained using any community detection algorithm. As a result, we prove that both estimators are strongly consistent with a penalty term of the same order as that of the penalized profile likelihood proposed by \cite{hu2020corrected}.
In terms of network sparsity, we proved that the three proposed estimators are strongly consistent in the regime where the connection probability scales as $\rho_n$, with $\rho_n \to 0$ and $n\rho_n \to \infty$. This is the most general sparsity regime under which consistency of estimators for the number of communities for the SBM has been established, as in \cite{cerqueira2020}. The previous works of \cite{wang2017likelihood} and \cite{hu2020corrected} proved consistency under the more restrictive regime where $n\rho_n / \log n \to \infty$. The sparsity regime under which we prove strong consistency of the proposed estimators includes the case where the connection probability scales as $\rho_n = \log n / n$, which coincides with the most general sparsity regime in which exact recovery of community assignments has been established for community detection methods \cite{Abbe2018}.

We also empirically evaluate the performance of the proposed estimator in comparison with other methods from the literature. Overall, the penalized DNML outperforms the others, particularly for unbalanced networks. In addition, we considered six benchmark network datasets to evaluate the performance of the proposed estimator on real-world network structures.

All codes for the DNML estimator are publicly available in the following GitHub repositories: \url{https://github.com/felipeb1914/ncmlnet} for the R language and \url{https://github.com/felipeb1914/ncml_py} for the Python language.

The remainder of the paper is organized as follows. In Section \ref{sec:model}, we present the SBM and introduce the notation used throughout the paper. In Section \ref{sec:model_selection}, we introduce the proposed estimators and establish the consistency of the estimator of the number of communities. In Section \ref{sec:computational}, we discuss the theoretical aspects related to the computation of the DNML and analyze its computational complexity. Simulation studies are carried out in Section \ref{sec:simulations} to investigate the empirical performance of the estimators. The analysis of real-world network data is presented in Section \ref{sec:applications}. Proofs of the main results are presented in the Appendix.

\section{Model}\label{sec:model}

In this section, we first introduce the stochastic block model (SBM). In this model, the existence of edges in the network depends on a latent partition of its vertices into communities. Let $\Z = (Z_1, \dots, Z_n)$ represent the division of the vertices into $k$ communities, such that the variables $Z_1, \dots, Z_n$ are independent and identically distributed and the probability that vertex $i$ belongs to community $a$, i.e., $Z_i = a$, is given by $\pi_a$, for $a \in [k]=\{1,\dots,k\}$ and $1 \leq i \leq n$. The edges in the network are Bernoulli random variables, conditionally independent given the communities. This model can be described as follows: 
\begin{equation}\label{eq:model_sbm}
\begin{split}
 \mathbb P(Z_i=a)&=\pi_a\,,\quad a=1,\dots,k \,\text{ and }\, 1\leq i \leq n \\
 A_{ij}\mid Z_i=a\,,\,Z_j=b & \sim \text{Bernoulli}(P_{ab})\,, \quad 1\leq i < j \leq n
\end{split} 
\end{equation}
for a given $k\times k$ symmetric matrix $P = (P_{ab})$. The probability vector $\pi$ and matrix $P$ are considered fixed and unknown.

We write the conditional distribution function of an observed network $\a$, given the communities assignments $\z\in [k]^n$, as
\begin{equation}\label{def-prob}
\P(\a\mid\z) = \prod_{1\leq a \leq b\leq k} \!\!P_{ab}^{o_{ab}(\a,\z)} (1-P_{ab})^{n_{ab}(\z)-o_{ab}(\a,\z)},
\end{equation}
and the distribution of the communities assignments $\z\in[k]^n$ as
\begin{equation}\label{def-prob-pi}
\Pp(\z) = \prod_{1 \leq a \leq k}\pi_a^{n_a(\z)},
\end{equation}
where the counts are given by
\begin{equation}\label{eq:count1}
     n_a(\z) = \sum_{i=1}^n \ind\{z_i=a\}, \quad 1\leq a \leq k
\end{equation}
\begin{equation}\label{eq:count3}
    o_{ab}(\a,\z) = \left\lbrace\begin{array}{cc}
        \sum\limits_{i,j=1}^n a_{ij}\ind\{z_i=a,z_j=b\}, &  a\neq b \\
         \sum\limits_{i < j}^n a_{ij}\ind\{z_i=a,z_j=b\}, &  a= b \\
    \end{array} \right.
\end{equation}
and
\begin{equation}\label{eq:count2}
     n_{ab}(\z) = \left\lbrace\begin{array}{cc}
        n_a(\z)n_b(\z), &  a\neq b \\
        \frac{1}{2}n_a(\z)(n_a(\z)-1), &  a= b \\
    \end{array} \right..
\end{equation}

The marginal likelihood function for an observed network $\a$ is given by
\begin{equation}\label{def:likelihood}
    \Pt(\a)=\sum\limits_{\e \in [k]^n}\Pt(\a , \e) = \sum\limits_{\e \in [k]^n}\P(\a \mid \e)\Pp(\e)\,.
\end{equation}

 Given a model with $k$ communities, define the parameter spaces $\mathcal{P}^k=\{P\in [0,1]^{k\times k}: P \text{ is symmetric}\}$ and $\Pi^k=\{\pi\in (0,1]^k : \sum_{a=1}^{k}\pi_a=1 \}$. The model selection problem refers to the process of choosing the appropriate number of communities $k$ in a SBM given an observed network. 
 
Moreover, in this work, we are interested in modeling sparse networks. In this context, we assume that the matrix $P$ may depend on the network size $n$ in such a way that its entries go to zero as $n \to \infty$. We reparameterize the matrix $P$ as $P = \rho_n S$, where $\rho_n \to 0$ as $n \to \infty$ and $S$ is fixed and does not depend on $n$.

\section{Model selection methods}\label{sec:model_selection}

Throughout this work, we assume that the pair $(\Z,\A)$ is generated from the SBM with an unknown $k_0$ communities and  unknown parameters $\pi_0 \in \Pi^{k_0}$ and $P_0 \in \mathcal{P}^{k_0}$. If the model has $k_0$ communities, it cannot be reduced to a model with fewer communities than $k_0$. This implies that $P_0$ does not have any identical columns.

To estimate the number of communities given a network $\A$, we define a penalized estimator of the form
\begin{equation}\label{eq:pen_estimator}
\hat{k}_n(\A) = \arg\max\limits_{1\leq k \leq n} \left\lbrace l_n(\A, k) - \text{pen}(k, n) \right\rbrace
\end{equation}
where $l_n(\textbf{A}, k)$ is a function of the network and $k$, and $\text{pen}(k,n)$ is a penalty function, which is increasing in both the number of vertices $n$ and $k$.

This penalized estimator is closely related to the Bayesian Information Criterion (BIC). In the classical BIC, the term $l(\mathbf{A},k)$ is taken to be the log-likelihood evaluated at the maximum likelihood estimator, while the penalty function $\text{pen}(k,n)$ is given by the number of free parameters in the model multiplied by the logarithm of the sample size.

Penalized estimators of the form \eqref{eq:pen_estimator} have been studied in the literature in recent years for model selection in the SBM. To estimate the number of communities \cite{wang2017likelihood} propose a penalized estimator based on the maximum likelihood function, that is, $l_n(\a, k) = \log\sup_{\pi \in  \Pi^{k}, P\in\mathcal{P}^{k} }\Pt(\a)$. They proved the consistency of this estimator, that is, the estimator converges in probability to the true number of communities as $n\to\infty$, considering the penalty function $\pen(k, n) = ({k(k + 1)/2}) n \log n$. 

In a later work by \cite{hu2020corrected}, the authors proved consistency in the case where $l_n(\a, k) = \log\sup_{\pi \in  \Pi^{k}, P\in\mathcal{P}^{k} }\max_{\z\in[k]^n}\Pt(\a|\z)$ is the profile log-likelihood function with $\pen(k, n) = ({k(k + 1)/2}) \log n + n \log k$. Note that, when using the profile log-likelihood, the order of the penalty function increases to a logarithmic term plus a linear term in $n$. Due to this additional term that is linear in $n$, the authors referred to this estimator as the corrected BIC estimator. The consistency of penalized estimators based on both the maximum likelihood and the profile likelihood has been established under the sparsity regime $\rho_n \to 0$ with $n\rho_n / \log n \to \infty$.

In \cite{cerqueira2020}, the authors proved consistency when $l_n(\a, k)$ corresponds to the integrated likelihood, obtained by integrating the likelihood function in \eqref{def:likelihood} over the parameters $\pi$ and $P$. The estimator is shown to be strongly consistent, that is, it equals the true number of communities eventually almost surely as $n\to \infty$, under a penalty function of order $k^3 \log n$. This estimator has also been studied for model selection for the Degree Corrected SBM by \cite{Cerqueira24_DCSBM} and for multilayer and dynamic  SBM by \cite{Arts2025}. With respect to sparsity, strong consistency of the penalized integrated likelihood estimator has been proved under a more general sparse regime, characterized by $\rho_n \to 0$ and $n\rho_n \to \infty$.

In this work, our interest lies in the study of estimators constructed from normalized likelihood-based quantities. Specifically, we investigate the consistency of three such estimators. The first one is based on the \textit{Normalized Maximum Likelihood} (NML), defined, for $\a \in \{0,1\}^{n\times n }$ as
\begin{equation}
    \NML_k(\a) = \dfrac{\sup\limits_{\pi \in  \Pi^{k}, P\in\mathcal{P}^{k} }\Pt(\a)}{\CNML(k,n)}
\end{equation}
where
\begin{equation}
    \CNML(n,k) = \sum\limits_{\x}\sup\limits_{\pi \in  \Pi^{k}, P\in\mathcal{P}^{k} }\Pt(\x).
\end{equation}

Penalized estimators based on the NML have been previously studied in other contexts. For instance, \cite{gassiat2003} proposed a penalized estimator derived from the normalized maximum likelihood function to estimate the number of hidden states in hidden Markov models. In \cite{yamanishi2019decomposed}, the authors empirically investigate the normalized-based estimators for latent variable models.

For the SBM, we define the penalized estimator based on the NML as
\begin{equation}\label{eq:est_NML}
    \hkNML(\A) = \arg\max_{1\leq k \leq n}\left\lbrace \log\NML_k(\A ) - \pen(k,n)\right\rbrace.
\end{equation}

An appropriate choice of the penalty function $\pen(k,n)$ is essential to guarantee the consistency of the proposed estimator. The penalty function that guarantees the consistency $\hkNML$ is of the form
\begin{equation}\label{eq:pen_NML}
\begin{split}
\pNML(k,n) &= \sum\limits_{i=1}^{k-1}\left( \frac{i(i+2)+1+\epsilon}{2} \right)\log n\\
&= \left( \frac{k(k-1)(2k-1)}{12} + \frac{(k-1)(k+1+\epsilon)}{2} \right)\log n
 \end{split}
\end{equation}
for any $\epsilon>0$.
This penalty function is of order $k^3 \log n$, which is of the same order as the penalty used in \cite{cerqueira2020} and smaller than those considered in \cite{wang2017likelihood} and \cite{hu2020corrected}. Our first result establishes strong consistency of the penalized NML estimator, that is, it correctly estimates the true number of communities almost surely as $n\to \infty$, under a more general sparsity regime, characterized by $\rho_n \to 0$ and $n\rho_n \to \infty$. The proof of this result is given in Appendix.

\begin{theorem}\label{teo:consistency_NML}
      Suppose that $\A$ is generated from a SBM with $k_0$ communities and parameters $\pi_0$ and $P_0=\rho_nS_0$ with either $\rho_n \to 0$ and $n\rho_n\to \infty$, or $\rho_n = 1$.  For the penalty function given by \eqref{eq:pen_NML}, for any $\epsilon>0$, we have that $\hkNML(\A)=k_0$     eventually almost surely as $n\to \infty$.
\end{theorem}

It is worth noting that the numerator of the NML corresponds to the maximum marginal likelihood, whose exact computation is infeasible since the marginal likelihood function, given by \eqref{def:likelihood}, involves a sum over all possible community assignments, which grows exponentially with the network size. However, in practice, this problem is commonly addressed using variational Expectation-Maximization (VEM) algorithms, which approximate the intractable marginal likelihood efficiently \citep{daudin2008mixture}.
Even though the marginal likelihood can be approximated through VEM algorithms, 
the computation of the normalization constant $\CNML$ remains infeasible, as it involves a sum over all possible networks. Consequently, although Theorem~\ref{teo:consistency_NML} guarantees strong consistency of the estimator under a more general sparsity regime, its practical implementation is infeasible.

To address this issue, one possible approach is to replace the marginal likelihood function with the complete likelihood function. In this way, we define the \textit{Normalized Maximum Complete Likelihood} (NMCL) for $\a\in\{0,1\}^n$ and $\z\in[k]^n$ as
\begin{equation}
    \NMCL_k(\a,\z) = \dfrac{\sup\limits_{\pi \in  \Pi^{k}, P\in\mathcal{P}^{k} }\Pt(\a,\z)}{\CNMCL(n,k)}
\end{equation}
where
\begin{equation}
    \CNMCL(n,k) = \sum\limits_{\e}\sum\limits_{\x}\sup\limits_{\pi \in  \Pi^{k}, P\in\mathcal{P}^{k} }\Pt(\x,\e).
\end{equation}

For any community assignment $\mathbf{z} \in [k]^n$, the maximum complete likelihood is easily computable, making the numerator of the NCML tractable. However, the normalizing constant $\CNMCL$ involves a sum over all possible community assignments of $n$ nodes into $k$ communities and over all possible networks, which also remains computationally infeasible.

While the NML criterion normalizes the maximum marginal likelihood over all possible networks, the NCML normalizes the maximum complete likelihood over all possible networks and community assignments. To account for the structure of the network $\a$ and the community assignment $\z \in [k]^n$ separately, \cite{yamanishi2019decomposed} proposed the \emph{Decomposed Normalized Maximum Likelihood} (DNML), which normalizes each component independently.
The DMNL is defined, for $\a \in \{0,1\}^{n\times n}$ and $\z\in[k]^n$, as

\begin{equation}\label{def:DNML}
    \DNML_k(\a, \z) =\DNML_k(\a\mid \z)\DNML_k(\z)
\end{equation}
where
\begin{equation}\label{def:DNML_A}
    \DNML_k(\a\mid \z) = \dfrac{\sup\limits_{P \in \mathcal{P}^{k}}\P(\a\mid \z)}{\CDMNL(\z,k)},
\end{equation}
and
\begin{equation}\label{def:DNML_z}
\DNML_k(\z) = \frac{\sup\limits_{\pi \in \Pi^{k}}\Pp(\z)}{\CDMNLz(n,k)},
\end{equation}
with 
\begin{equation}
    \CDMNL(\z,k) = \sum\limits_{\x}\sup\limits_{P \in \mathcal{P}^{k}}\P(\x\mid \z),
\end{equation}
and
\begin{equation}
\CDMNLz(n,k) = \sum\limits_{\e}\sup\limits_{\pi \in \Pi^{k}}\Pp(\e).
\end{equation}
In this case, observe that the normalizing term $\CDMNL(\z,k)$ is taken over all possible networks conditional on the community assignment $\z$, while the normalizing term $\CDMNLz(n,k)$ is taken over all networks with $n$ nodes. By this definition, DNML satisfies
\begin{equation}
\sum_{\z} \sum_{\a}\DNML_k(\a \mid \z)\,\DNML_k(\z) = 1.
\end{equation}

The key difference between the NML, the NMCL, and the DNML lies in the way normalization is performed. The NML and NMCL normalize over the entire network space at once, whereas the DNML performs the normalization separately for each component (the network and the community assignments). An important advantage of the DNML is that, by decomposing the normalization into two terms, $\CDMNL$ and $\CDMNLz$, it becomes computationally feasible, as discussed in \cite{yamanishi2019decomposed}. The computation of the DNML is presented in Section \ref{sec:computational}.

Both the NMCL and the DNML depend on the community assignments $\mathbf{z} \in [k]^n$. To define the model selection criterion, rather than calculating the NMCL and DNML for all possible community assignments, 
we consider the estimated communities $\hz \in [k]^n$. To this end, we define
\begin{equation} \label{eq:estimate_z}
    \hz = \arg\max\limits_{\z \in [k]^n} \left\lbrace\sup\limits_{{\pi \in  \Pi^{k}, P\in\mathcal{P}^{k}}}\Pt(\a, \z) \right\rbrace.
\end{equation}

It is important to notice that $\hz$ depends on the network $\a$. This dependence is omitted throughout the paper for clarity, and will be explicitly indicated whenever it is relevant.

Finally, we define the penalized estimators based on NMCL and DNML by
\begin{equation}\label{eq:est_NMCL}
    \hkNMCL(\A) = \arg\max_{1\leq k \leq n}\left\lbrace \log\NMCL_k(\A, \hz ) - \pen(k,n)\right\rbrace,
\end{equation}
and
\begin{equation}\label{eq:est_DNML}
    \hk(\A) = \arg\max_{1\leq k \leq n}\left\lbrace \log\DNML_k(\A, \hz ) - \pen(k,n)\right\rbrace.
\end{equation}

To properly choose the penalty function in order to guarantee the consistency of the estimators $\hkNMCL$ and $\hk$, we add an extra term to the penalty function \eqref{eq:pen_NML} used for the NML estimator, which is linear in the network size. Thus, we define
\begin{equation}\label{eq:pen_DNML}
\begin{split}
\pDNML(k,n) &= \sum\limits_{i=1}^{k-1}\left( \frac{i(i+2)+1+\epsilon}{2} \right)\log n + n\sum\limits_{i=1}^{k-1}\log(i) \\
&= \left( \frac{k(k-1)(2k-1)}{12} + \frac{(k-1)(k+1+\epsilon)}{2} \right)\log n + n\log((k-1)!)
\end{split}
\end{equation}
for any $\epsilon>0$. This choice of an appropriate form of the penalty function is crucial to prove that the proposed estimators do not overestimate the true number of communities, as discussed in Section \ref{sec:non_overestimation}.

The resulting penalty has the same order as that considered in \cite{hu2020corrected}, since the community assignments are estimated and incorporated into the estimator. In both cases, accounting for the uncertainty associated with the latent labels introduces an additional term linear in the network size, leading to penalties of comparable order.

The following result shows that both estimators, $\hk$ and $\hkNMCL$, are strongly consistent under the same general sparsity regime as the NML estimator.

\begin{theorem}\label{teo:consistency_DNML}
      Suppose that $\A$ is generated from a SBM with $k_0$ communities and parameters $\pi_0$ and $P_0=\rho_nS_0$ with either $\rho_n \to 0$ and $n\rho_n\to \infty$, or $\rho_n = 1$.  For the penalty function given by \eqref{eq:pen_DNML}, for any $\epsilon>0$, we have that $\hk(\A)=k_0$ and $\hkNMCL(\A)=k_0$ eventually almost surely as $n\to \infty$.
\end{theorem}

Even though the estimators $\hkNML$ and $\hkNMCL$ are theoretically strongly consistent for a general sparsity regime, their computation is infeasible in practice due to the intractability of the associated normalization constants. In contrast, the estimator $\hk$ enjoys the same theoretical guarantees, including strong consistency, while remaining computationally feasible, and can be efficiently implemented as described in Section~\ref{sec:computational}.

The proofs of Theorems \ref{teo:consistency_NML} and \ref{teo:consistency_DNML} are divided in two parts. First, Proposition \ref{prop:non_overestimation_DNML} and Proposition \ref{prop:non_overestimation_NML} guarantee that the proposed estimators do not overestimate the true number of communities $k_0$. In the second part, Proposition~\ref{prop:non_underestimation} states that these estimators do not underestimate the true number of communities. In fact, using a penalty term of order $\log n$ is sufficient to ensure that all the proposed estimators do not underestimate the true number of communities. However, the additional term in the penalty \eqref{eq:pen_DNML} is included to guarantee that the DNML and NMCL estimators do not overestimate the true number of communities (see Propositions \ref{prop:non_overestimation_DNML} and \ref{prop:non_overestimation_NML}). In the simulation studies presented in Section \ref{sec:simulations}, we further discuss the impact of the extra term in the penalty function on the performance of the DNML estimator. As expected, the addition of a term of order $n$ in the penalty function \eqref{eq:pen_DNML}  tends to make the estimator underestimate the true number of parameters, as it favors simpler models

\section{Computational properties of the DNML}\label{sec:computational}

As discussed before, the exact computation of the NML and NMCL estimators is intractable, as it involves summing over all possible networks and community assignments. Although the DNML estimator, proposed in \eqref{eq:est_DNML}, involves computing $\CDMNL$ and $\CDMNLz$, which require summing over all possible networks and all possible community assignments, it has been shown to be computationally feasible by \cite{yamanishi2019decomposed}. For the sake of completeness, in this work we also develop how to compute this quantity based on the procedure proposed by \cite{yamanishi2019decomposed}. It is important to note that they proposed this calculation for latent variable models in general; however, in the case of the SBM, no theoretical results regarding the consistency of the penalized estimator of the form \eqref{eq:est_DNML} have been established.

For any $\e\in [k]^n$, define the normalizing term of the of the maximum likelihood for a multinomial distribution with $n$ trials and $k$ categories as
\[C_{\text{MN}}(n,k) = \sum\limits_{\e}\prod_{a=1}^k\left( \frac{n_a(\e)}{n}\right)^{n_a(\e)}.\]
For any $\e\in[k]^n$, using the fact that the maximum likelihood estimators of $\pi_a$ is given by $n_a(\e)/n$, we have that
\begin{equation}\label{eq:relation_CMN_z}
    \CDMNLz(n,k) = C_{\text{MN}}(n,k).
\end{equation}

In \cite{Kontkanen2007}, the authors proved that, in the general case, the normalizing constant for the maximum likelihood estimate of a multinomial distribution with $Q$ components and a total of  $m$ trials can be computed recursively by
\begin{equation}
    \begin{split}
        &C_{\text{MN}}(m,1)=1,\\
        & C_{\text{MN}}(m,2) = \sum_{t=0}^{m} \frac{m!}{t!(m-t)!}\left(\frac{t}{m}\right)^{t}\left(\frac{m-t}{m}\right)^{m-t},\\
        &C_{\text{MN}}(m,Q) = C_{\text{MN}}(m,Q-1) +\frac{m}{Q-2}C_{\text{MN}}(m,Q-2).
    \end{split}
\end{equation}

Computing $C_{\text{MN}}(m,2)$ requires a sum of length $m$, yielding a linear cost in $m$. In that way, when calculating $C_{\text{MN}}(m,q)$ for $q > 2$ until $q = Q$, the cost of each iteration takes a constant time since the values of $C_{\text{MN}}(m,q-1)$ and $C_{\text{MN}}(m,q-2)$ have already been calculated in previous steps. Thus, the quantity $C_{\text{MN}}(m,Q)$ is computable in time $O(m + Q)$, implying that $\CDMNLz(n,k)$ is computable in time $O(n+k)$.

The next result states that the normalizing constant $\CDMNL$ can be expressed as a product over pairs of communities of the normalizing constant of a multinomial distribution.  In this way, we reduce the problem of computing 
$\CDMNL$ from a sum over all networks to a product over all pairs of communities, which corresponds to a much smaller space. The next result also states that $\CDMNL$ can be computable in linear time with respect to $n_{ab}$.  This result is proved using the same approach as in \cite{yamanishi2019decomposed}.

\begin{proposition}\label{prop:aprox_C}
    For any $\z\in [k]^n$, we have that
    \begin{equation}
         \CDMNL(\z,k) = \prod\limits_{a \leq b}C_{\text{MN}}(n_{ab}(\z),2),
    \end{equation}
where
\begin{equation}
    C_{\text{MN}}(m,2) = \sum_{t=0}^{m} \frac{m!}{t!(m-t)!}\left(\frac{t}{m}\right)^{t}\left(\frac{m-t}{m}\right)^{m-t}.
\end{equation}
\end{proposition}

\begin{proof}
 First,  for any network $\x$, define the subnetwork $\mathbf{x}_{(ab)}=\{x_{ij}: z_i=a, z_j=b\}$ of $\x$, that is, $\mathbf{x}_{(ab)}$ is the subnetwork of $\x$ consisting of the set of pairs of vertices that belongs to community $a$ and $b$. 
 
 Since $o_{ab}(\mathbf{x}_{(ab)},\z)$ counts the number of edges out of the total $n_{ab}(\z)$ possible edges, the quantity $C_{\text{MN}}(n_{ab}(\z),2)$ can be interpreted as the normalizing term of the maximum likelihood for a multinomial distribution with two categories (edges and non-edges) and a total of $n_{ab}(\z)$ trials. Thus, for any community assignment $\z\in [k]^n$, observe that, for $1 \leq a\leq b\leq k$,
\begin{equation}
  \begin{split}
        C_{\text{MN}}(n_{ab}(\z),2) &= \sum_{t=0}^{n_{ab}(\z)} \frac{n_{ab}(\z)!}{t!(n_{ab}(\z)-t)!}\left(\frac{t}{n_{ab}(\z)}\right)^{t}\left(1-\frac{t}{n_{ab}(\z)}\right)^{n_{ab}(\z)-t} \\
        & =\sum\limits_{\mathbf{x}_{(ab)}}  \left( \dfrac{o_{ab}(\mathbf{x}_{(ab)},\z)}{n_{ab}(\z)} \right)^{o_{ab}(\mathbf{x}_{(ab)},\z)}\left(1- \dfrac{o_{ab}(\mathbf{x}_{(ab)},\z)}{n_{ab}(\z)} \right)^{n_{ab}(\z)-o_{ab}(\mathbf{x}_{(ab)},\z)}.
  \end{split}
\end{equation}

For a given $\z\in[k]^n$, the conditional maximum likelihood estimators of $P_{ab}$ is given by  $o_{ab}(\a,\z)/n_{ab}(\z)$ and we can write
\begin{equation}\label{eq:max_likelihood2}
    \sup_{P \in \mathcal{P}^{k}} \P(\a\mid \z)  = \prod\limits_{a \leq b}  g_{ab}(\x,\z).
\end{equation}
with
\[g_{ab}(\x,\z) =  \left( \dfrac{o_{ab}(\a,\z)}{n_{ab}(\z)} \right)^{o_{ab}(\a,\z)}\left(1- \dfrac{o_{ab}(\a,\z)}{n_{ab}(\z)} \right)^{n_{ab}(\z)-o_{ab}(\a,\z)}.\]
Thus,
    \begin{equation}
    \CDMNL(\z,k) = \sum\limits_{\x}\prod\limits_{a \leq b}   g_{ab}(\x,\z).
\end{equation}

For a fixed $\z\in[k]^n$, let $E_{ab}(\z)=\{\{i,j\}:z_i=a,\,z_j=b\}$ denote the set of potential edges between communities $a$ and $b$. The family $\{E_{ab}(\z)\}_{a\le b}$ forms a partition of the edge set, so every adjacency matrix $\x\in\{0,1\}^{n\times n}$ can be uniquely written as the tuple of subconfigurations $(\mathbf{x}_{(ab)})_{a\le b}$, where $\mathbf{x}_{(ab)}=(x_e)_{e\in E_{ab}}$. Note that, for fixed $\z\in[k]^n$, the quantity $o_{ab}$ depends on $\x$ only through the subnetwork $\mathbf{x}_{(ab)}$, while $n_{ab}$ depends only on $\z$. Thus we can write $g_{ab}(\x,\z)=g_{ab}(\mathbf{x}_{(ab)},\z)$. Hence
\begin{equation}
    \begin{split}
        \CDMNL(\z,k) &= \sum\limits_{\mathbf{x}_{(11)}}\sum\limits_{\mathbf{x}_{(12)}}\dots \sum\limits_{\mathbf{x}_{(kk)}}\prod\limits_{a \leq b} g_{ab}(\mathbf{x}_{(ab)},\z)  \\
        & =\prod\limits_{a \leq b}\sum\limits_{\mathbf{x}_{(ab)}} g_{ab}(\mathbf{x}_{(ab)},\z)\\
        & = \prod\limits_{a \leq b}C_{\text{MN}}(n_{ab}(\z),2).
    \end{split}
\end{equation}

\end{proof}

By combining Proposition \ref{prop:aprox_C} and Equation \eqref{eq:relation_CMN_z}, the logarithm of the DNML, defined in \eqref{def:DNML}, can be written as
\begin{equation}\label{eq:DNML_final}
\begin{split}
\log\DNML_k(\a\mid \z) &= \sum_{a=1}^n n_{a}(\z)\log\left(\frac{n_{a}(\z)}{n}\right) +
\sum\limits_{a \leq b}  {o_{ab}(\a,\z)}\log\left( \dfrac{o_{ab}(\a,\z)}{n_{ab}(\z)} \right) \\
& + \sum\limits_{a \leq b}{(n_{ab}(\z)-o_{ab}(\a,\z))}\log\left(1- \dfrac{o_{ab}(\a,\z)}{n_{ab}(\z)} \right)  \\
& -\log C_{\text{MN}}(n,k) - \sum\limits_{a \leq b}\log C_{\text{MN}}(n_{ab}(\z),2).
\end{split}
\end{equation}
 
The pseudo-code of the algorithm to estimate the number of communities using the DNML estimator is given in Algorithm \ref{algorithm}.\\

\begin{algorithm}[H] \label{algorithm}
\caption{Model Selection Algorithm using the Penalized DNML Criterion}

\textbf{Input:} adjacency matrix $\mathbf{A}_{n \times n}$, maximum number of communities $K_{\max}$.

\begin{enumerate}
\item For each candidate $1 \le k \le K_{\max}$:
    \begin{enumerate}
        \item[(1.1)] \textbf{(Label assignment)} compute the community labels 
        $\hat{\mathbf z}_n$ with $k$ groups in \eqref{eq:estimate_z}.
        
        \item[(1.2)] \textbf{(Parameter estimation)} estimate the SBM parameters 
        $\hat{\theta} = (\hat{\pi}, \hat{P})$ from $\hat{\mathbf z}_n$.
        
        \item[(1.3)] \textbf{(Counts)} compute the counts 
        $n_a$ \eqref{eq:count1}, $n_{ab}$ \eqref{eq:count2} and $o_{ab}$ \eqref{eq:count3} induced by $\hat{\theta}$.
        
        \item[(1.4)] \textbf{(Criterion evaluation)} compute the $\DNML_k(\a, \z)$ in \eqref{eq:DNML_final} and the $\pDNML(k,n)$ in \eqref{eq:pen_DNML}.
    \end{enumerate}

\item Compute $\hk(\A)$ in \eqref{eq:est_DNML}.
\end{enumerate}

\textbf{Output:} the estimated number of communities $\hk$.
\end{algorithm}

\vspace{0.5cm}

Taking $K_{\max}=n$, the penalized estimator has cubic computational complexity, as the DNML criterion is evaluated for each $k = 1,\dots,n$, with a cost of $O(n^2)$ per evaluation. This computational cost does not take into account the computation of $\hz$ at each iteration, which in practice can be obtained using any community detection method. If the algorithm used to estimate $\hz$ has complexity at most $O(n^2)$, then the total computational cost is of order $O(n^3)$. Otherwise, if the community detection algorithm has complexity $O(w)$ with $w > n^2$, the overall complexity of the penalized DNML algorithm becomes $O(w \cdot n)$. It is important to emphasize that if the algorithm is evaluated for $k$ varying up to a fixed $K_{\max}$, as described in Algorithm \ref{algorithm} and commonly done in the literature, the total computational cost of $O(n^2)$.

To study the computational cost, we compute it using the Spectral Clustering \citep{Lei_SC_2015} and  Fast-Greedy \citep{clauset2004finding} algorithms for community detection. In Figure \ref{fig:tempot}, we observe that the execution time excluding the community detection step (red line) is well fitted by a third-degree polynomial. Moreover, the curves corresponding to the total computation time of the penalized estimator (green line) differ because the community detection methods have different execution times. In fact, the total execution time of the algorithm is higher when using Spectral Clustering, since this community detection method has computational complexity $O(n^3)$. In contrast, the Fast Greedy method exhibits approximately $O(n \log^2 n)$ complexity, making it computationally less demanding.

\begin{figure}[H]
    \centering
    \begin{subfigure}{0.49\textwidth}
        \centering
        \includegraphics[width=\linewidth]{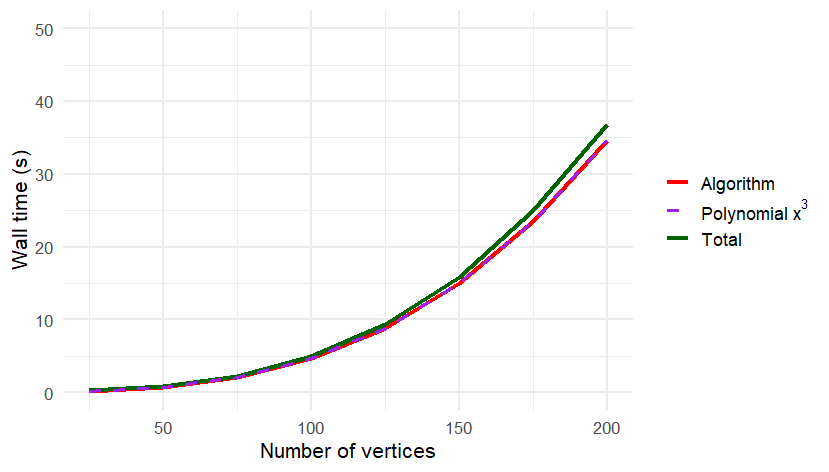}
        \caption{\textit{Fast-Greedy}}
    \end{subfigure}
    \hfill
    \begin{subfigure}{0.49\textwidth}
        \centering
        \includegraphics[width=\linewidth]{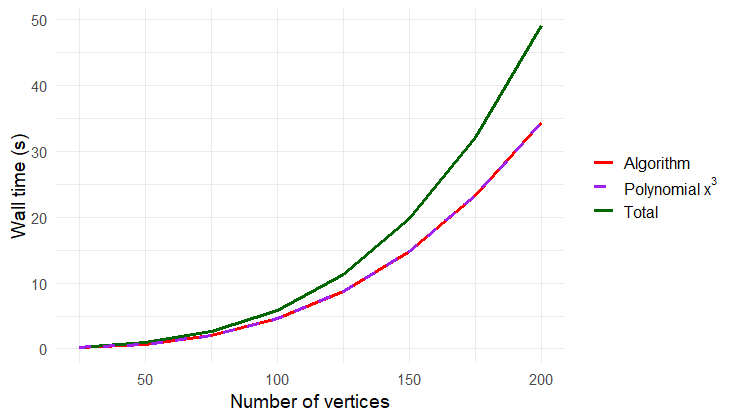}
        \caption{\textit{Spectral Cluster}}
    \end{subfigure}
    \caption{Comparison of the average total execution time in seconds of the penalized estimator excluding the estimation time of $\hz$ (red line) and including the estimation time of $\hz$ (green line), using the Fast-Greedy and Spectral Clustering algorithms.}
    \label{fig:tempot}
\end{figure}

The variational approach proposed by \cite{daudin2008mixture} for estimating the communities in \eqref{eq:estimate_z} becomes increasingly computationally demanding as $n$ grows. In fact, even though we run Algorithm \ref{algorithm} with a fixed $K_{\max}=10$, Table~\ref{tab:time_table} shows that the total running time of the algorithm is approximately twice the time required to run spectral clustering with $K_{\max}=n$. 
However, the runtime for computing the DNML remains low.

\begin{table}[H]
\centering
\caption{Wall time (in seconds) of Algorithm~\ref{algorithm} when $\hz$ is estimated using the variational algorithm. }
\setlength{\tabcolsep}{4pt}
\begin{tabular}{c|c|c|c|c|c|c|c|c}
\hline
$n$ & {25} & {50} & {75} & {100} & {125} & {150} & {175} & {200} \\ \hline
Wall time Total (s)& 9.016 & 10.404 & 12.542 & 23.237 & 33.374 & 52.003 & 79.184 & 103.090 \\ \hline
Wall time DNML & 0.035 & 0.047 & 0.068 & 0.104 & 0.171 & 0.235 & 0.327 & 0.420 \\ \hline
\end{tabular}
\label{tab:time_table}
\end{table}

\section{Empirical Analysis}\label{sec:simulations}

In this section, we study the performance of the penalized DNML estimator on synthetic data. We compare the proposed penalized DNML estimator with other approaches based on different principles. These include methods based on cross-validation over vertices (NCV) \citep{Chen2028_NCV} and over edges (ECV) \citep{li2020network}, as well as approaches based on spectral properties of graphs, such as hypothesis testing (GFit) \citep{Lei2016_test} and the Bethe Hessian matrix (BHMC) \citep{le2022estimating}. In addition, we consider likelihood-based methods, including penalized maximum likelihood (PML) \citep{wang2017likelihood}, penalized profile likelihood, also referred to as corrected BIC (CBIC) \citep{hu2020corrected}, and integrated likelihood (IL) \citep{cerqueira2020}. Both the PML and CBIC approaches involve tuning parameters, which are set to 0.04 and 1, respectively. For the DNML we used the penalty function \eqref{eq:pen_DNML} with $\epsilon=0.5$.

The community detection method used to obtain the plug-in estimates for CBIC and DNML is spectral clustering \citep{Lei_SC_2015}. We opted to use spectral clustering instead of the VEM algorithm due to its computational cost, as discussed in Section \ref{sec:computational}. In fact, both methods performed well in estimating the communities, without affecting the performance of the estimators for the number of communities.

In this study, we focus on three scenarios to evaluate the performance of the proposed estimator in comparison with the others. First, it is of interest to analyze how the methods perform as the number of vertices $n$ increases. Next, fixing the number of vertices, we study how varying the absolute difference between within-community and between-community edge probabilities affects the methods’ ability to distinguish communities, particularly when these probabilities are nearly equal. Finally, we consider a sparsity scenario to evaluate the performance of the proposed estimator under a sparsity regime.

For each scenario and model setting, 100 networks are generated from the SBM model with fixed $k_0$. For each model selection method, the average estimated value of $k_0$ is used to evaluate the performance of the algorithms. To control the simulation parameters, we fix the within-community edge probability at $a$ and the between-community edge probability at $b$.

 In the first scenario, we consider a within-community edge probability of $a = 0.8$ and a between-community edge probability of $b = 0.3$. The true number of communities is fixed at $k_0 = 5$ and we vary the number of vertices. We consider two settings within this scenario: a balanced network ($\pi=(1/5,1/5,1/5,1/5,1/5)$) and an unbalanced network ($\pi=(0.6, 0.1, 0.1, 0.1, 0.1)$). Figure \ref{fig:sim1} shows that the mean estimated number of communities obtained by the penalized DNML estimator approaches the true number of communities as the number of vertices $n$ increases. For small values of $n$, the proposed method tends to underestimate the true number of communities. This behavior may be explained by the linear dependence of the penalty term on $n$. When $n$ is small, the penalty can dominate the gain in likelihood obtained by increasing the number of communities, leading the criterion to favor simpler models. However, for unbalanced networks, the penalized DNML converges to the true number of communities faster than the others methods.

\begin{figure}[H]
    \centering
    \begin{subfigure}{0.49\textwidth}
        \centering
        \includegraphics[width=\linewidth]{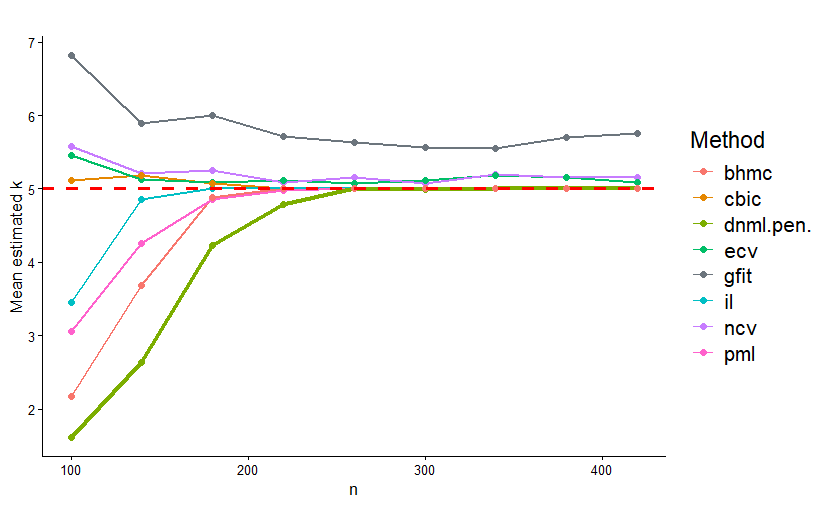}
        \caption{Balanced Network}
    \end{subfigure}
    \hfill
    \begin{subfigure}{0.49\textwidth}
        \centering
        \includegraphics[width=\linewidth]{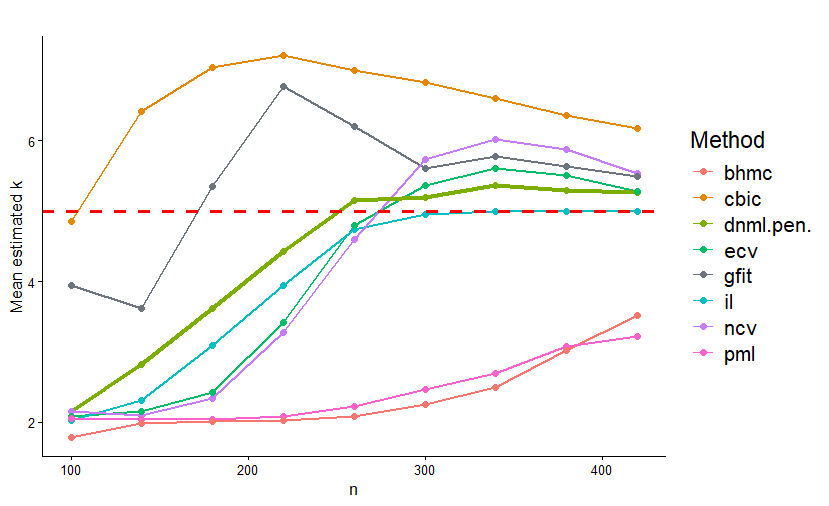}
        \caption{Unbalanced network}
    \end{subfigure}
    \caption{Average estimated number of communities for each method as the number of vertices $n$ varies, with $a = 0.8$, $b = 0.3$, and $k_0 = 5$, under balanced and unbalanced network settings.}
    \label{fig:sim1}
\end{figure}

 Next, we fix the number of vertices in the network at $n = 200$ and then vary the between-community edge probabilities. For the balanced setting, the edge probability within communities is set to $a = 0.9$, and for the unbalanced setting ($\pi = (0.6, 0.1, 0.1, 0.1, 0.1)$), it is set to $a = 0.8$. By Figure \ref{fig:sim2}, as expected, as the absolute difference between the edge probabilities within and between communities decreases, i.e., as the value of $b$ increases, the performance of all estimators decreases. In the balanced setting, the penalized DNML deviates from the true number of communities, $k_0 = 5$, faster than the competing methods as $b$ approaches $a$. In contrast, in the unbalanced setting, the penalized DNML performs better, with the mean estimated number of communities remaining closer to the true value even for larger values of $b$.

\begin{figure}[H]
    \centering
    \begin{subfigure}{0.49\textwidth}
        \centering
        \includegraphics[width=\linewidth]{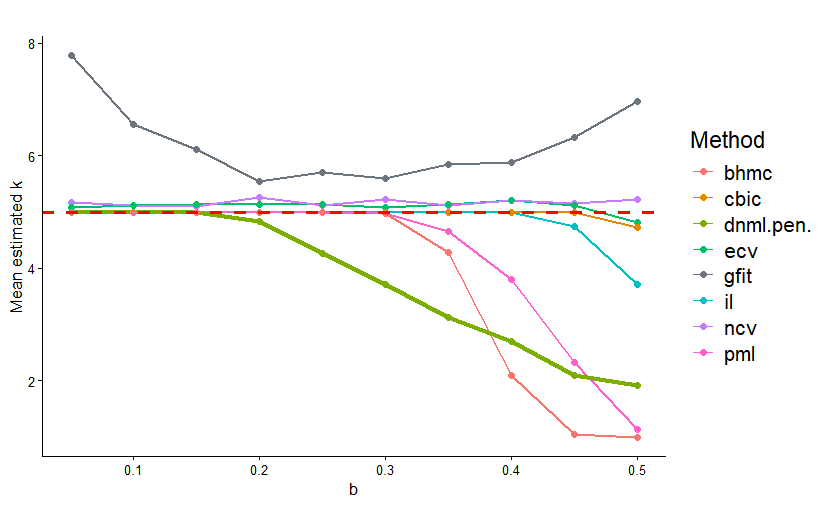}
        \caption{Balanced Network}
    \end{subfigure}
    \hfill
    \begin{subfigure}{0.49\textwidth}
        \centering
        \includegraphics[width=\linewidth]{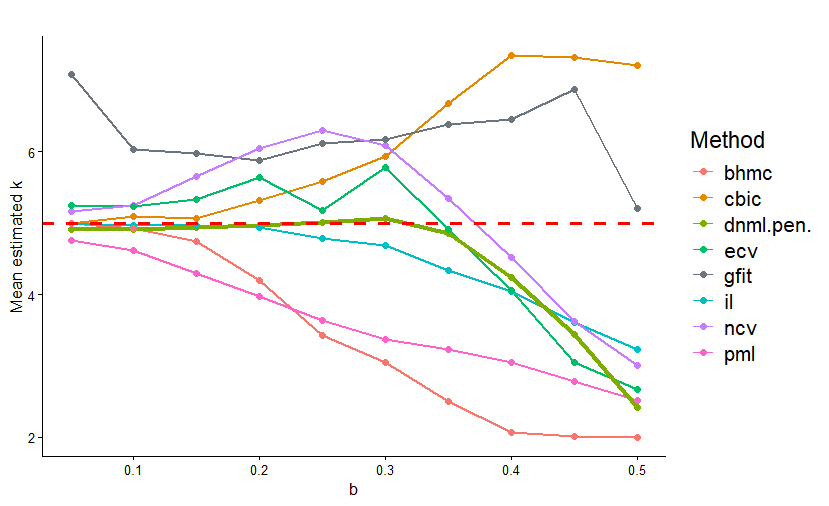}
        \caption{Unbalanced network}
    \end{subfigure}
    \caption{Mean estimated $k$ by the methods varying the probability $b$ of the network fixing $n = 200$, $k_0 = 5$, $a=0.8$ for balanced and $a=0.9$ unbalanced networks.}
    \label{fig:sim2}
\end{figure}

In order to study the impact of sparsity on the performance of the algorithm, we consider a network with $k_0=5$, $n=300$, $a=5$, and $b=1$. Figure~\ref{fig:sim_sparsity} shows that the DNML algorithm performs better than the other algorithms in the case of unbalanced communities ($\pi = (0.6, 0.1, 0.1, 0.1, 0.1)$). The IL method is the only one that outperforms DNML for unbalanced networks, which may be due to the fact that the penalty term used in IL is of smaller order.

\begin{figure}[H]
    \centering
    \begin{subfigure}{0.49\textwidth}
        \centering
        \includegraphics[width=\linewidth]{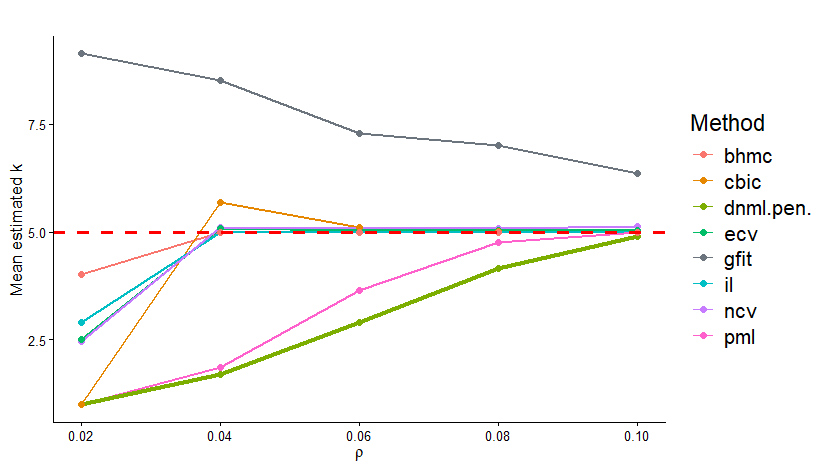}
        \caption{Balanced Network}
    \end{subfigure}
    \hfill
    \begin{subfigure}{0.49\textwidth}
        \centering
        \includegraphics[width=\linewidth]{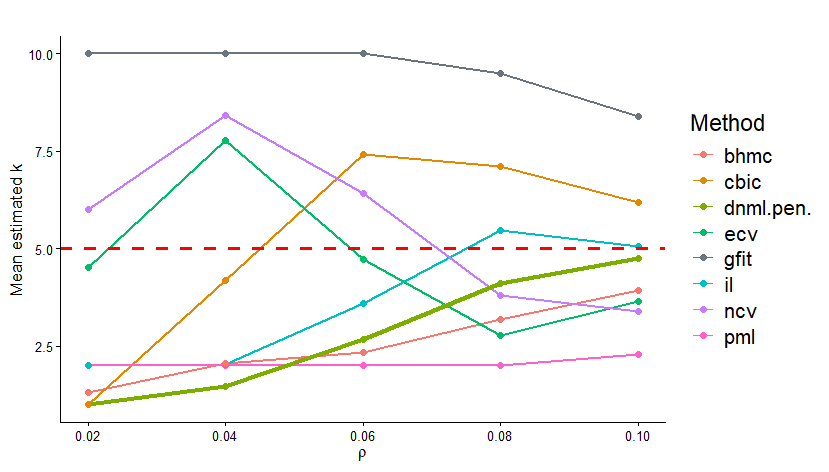}
        \caption{Unbalanced Network}
    \end{subfigure}
    \caption{Mean estimated $k$ by the methods varying $\rho_n$ fixing $a = 5$, $b=1$ and $k_0 = 5$.}
    \label{fig:sim_sparsity}
\end{figure}

Finally, we compare the performance of the different versions of the DNML estimator under distinct conditions. Specifically, we consider: (i) the DNML defined in \eqref{def:DNML} without the penalty term; (ii) the penalized DNML defined in \eqref{eq:est_DNML}, with the penalty term given in \eqref{eq:pen_DNML}; and (iii) a modified version of the penalized DNML that excludes the $n \log((k-1)!)$ term from the penalty function, which we refer to as the penalized log-DNML. First, we vary the number of vertices in the network $n$ for a balanced network with $a = 0.8$ and $b = 0.3$. Next, we consider a fixed number of vertices, $n = 120$, and $a = 0.9$, and we vary the between-community edge probability $b$. As shown in Figure \ref{fig:sim4}, the penalized DNML eventually approaches the true number of communities in both scenarios. However, this method tends to underestimate the true number of communities, in contrast to the DNML estimator without the penaty term, which tends to overestimate it. The penalized log-DNML, however, yields estimates of the number of communities that are closer to the true value, $k_0 = 5$, than the other two DNML-based estimators. It converges more rapidly to the true value in the first scenario and deviates more slowly from the true value in the second scenario. This is expected because the necessity of the extra term in the penalty function of the DNML estimator arises from Proposition \ref{prop:non_overestimation_DNML}, which guarantees that the DNML estimator does not overestimate the true number of communities.

\begin{figure}[H]
    \centering
    \begin{subfigure}{0.49\textwidth}
        \centering
        \includegraphics[width=\linewidth]{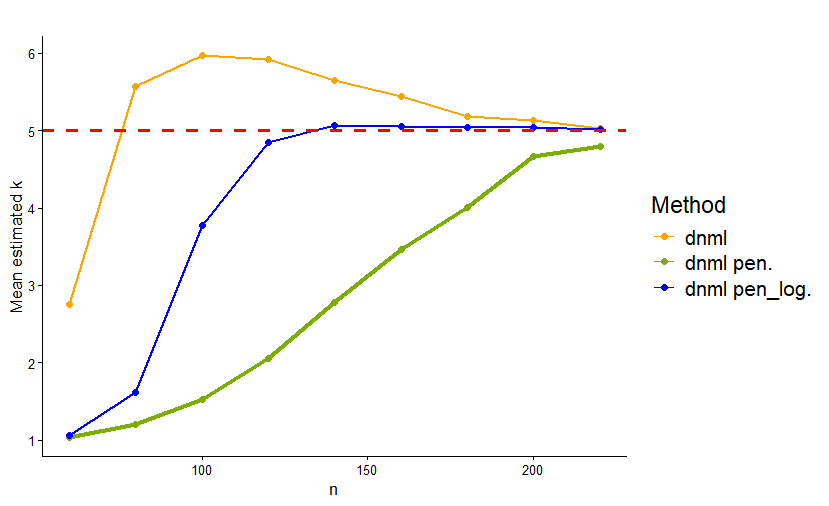}
        \caption{$a = 0.8$, $b = 0.3$}
    \end{subfigure}
    \hfill
    \begin{subfigure}{0.49\textwidth}
        \centering
        \includegraphics[width=\linewidth]{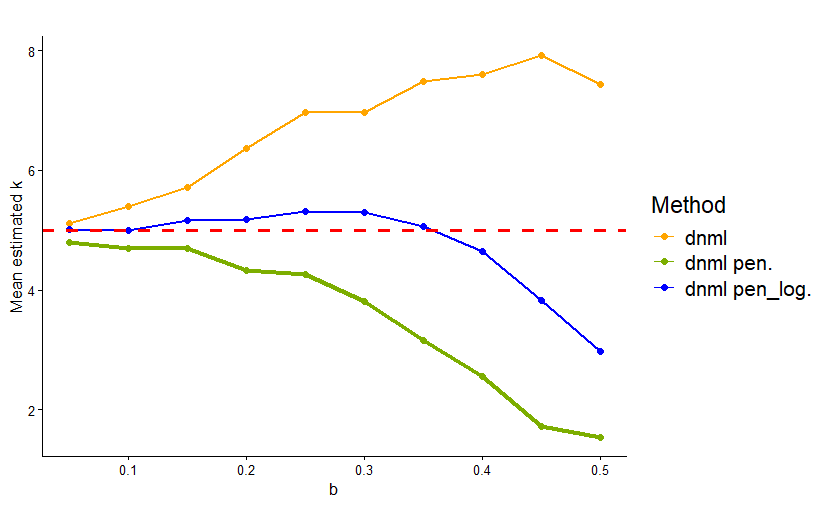}
        \caption{$n = 120$, and $a = 0.9$}
    \end{subfigure}
    \caption{Comparison of the different DNML-based estimators, varying the number of nodes $n$ (left) and varying the between-community edge probability $b$ (right) for balanced networks. Results are averaged over 100 replications.}
    \label{fig:sim4}
\end{figure}

\section{Applications}\label{sec:applications}

In this section, we investigate the performance of the proposed method in estimating the true number of communities, $k_0$, in real-world network data. We consider datasets that are commonly used in the model selection literature for random networks. These datasets serve as benchmark examples for evaluating the performance of the proposed methods, as they are widely studied and well established in the field.

The first dataset contains information on political books sold on Amazon in 2004, where an edge connects two books if they share the same purchasing trends. \citep{newman2006finding}. The books are classified into three categories: neutral, liberal, and conservative, corresponding to three communities in the network. The second dataset corresponds to a dolphin network in Doubtful Sound, New Zealand. An edge connects two dolphins if a social relationship between them is observed \citep{lusseau2003bottlenose}. The network contains two groups of dolphins, resulting in a ground-truth partition with two communities. The third dataset contains a friendship network in a karate club \citep{zachary1977information}, consisting of 34 members organized into two communities. The fourth dataset is a college football network, where each node represents a team and an edge indicates that the teams played a game against each other during the 2000 season. The network includes 115 teams divided into 12 conferences, which constitute the ground-truth communities \citep{girvan2002community}. In addition to this network, since it contains several communities with a small number of vertices, we artificially increased the number of vertices based on the SBM fitted to this network, up to 1150 vertices. This network was constructed by first estimating the community assignments using Spectral Clustering and then estimating the SBM parameters via maximum likelihood. A new network with 1,150 nodes was subsequently generated from the SBM using the estimated parameters. Finally, the last dataset is a network of 1222 political blogs divided into two groups: liberal and conservative \citep{adamic2005political}.

Considering all real-data applications jointly, the penalized DNML estimator proposed in this work demonstrates overall satisfactory performance in estimating the number of communities. Its results remain stable across networks with different structural characteristics, as reported in Table \ref{tab:resumo}. The DNML penalized method estimates the correct number of communities for the Dolphins, Blogs and Football (1150). In the college football dataset, a limitation of the DNML method becomes evident when the network contains a relatively large number of communities relative to the number of vertices. In this setting, the estimator struggles to accurately identify the true number of groups. However, when the same structural pattern is preserved and the number of vertices is synthetically increased through an SBM (Football (1150)), the estimator is able to recover the true number of communities exactly. In the political blogs dataset, all other estimators produced results that deviated noticeably from the ground truth, whereas the penalized DNML estimator recovered the exact number of communities. For cases in which the estimator failed to recover the true number of communities, it tended to underestimate the actual number. This behavior is likely due to the penalization term, which favors simpler models and may lead to selecting fewer communities than present in the data.

\begin{table}[H] 
\centering 
\caption{Estimated number of communities ($k$) obtained by each method for the different benchmark datasets.}
\resizebox{\textwidth}{!}{%
\begin{tabular}{lccccccccc} 
\toprule
\textbf{Method} & Truth & DNML Pen & NCV & ECV & PML & GFit & BHMC & IL & CBIC\\
\midrule
Books & 3 & 2  & 11 & 11 & 3 & 12 & 4 & 3 & 4\\
Dolphins & 2& 2  & 3 & 5 & 2 & 7 & 2 & 1 & 2\\
Karate & 2& 1 & 2 & 9 & 1 & 6 & 2 & 1 & 1\\
Football &  12 & 3  & 11 & 11 & 3 & 12 & 10 & 4 & 11\\
Football (1150) & 12 & 12 & 14 & 12 & 11 & - & 11 & 10 & 12\\
Blogs & 2& 2 & 20 & 20 & 4 & - & 8  & 10 & 1\\
\bottomrule
\end{tabular}%
}
\label{tab:resumo}
\end{table}

\section{Discussion}

In this work, we have studied penalized estimators for determining the number of communities in SBM, with a focus on approaches based on normalized likelihoods. Our theoretical results demonstrate that the proposed NML, NMCL, and DNML estimators are strongly consistent under a general sparsity regime, which includes the challenging setting where the average degree grows slowly with the network size. This extends previous consistency results for likelihood-based estimators and establishes strong guarantees for estimators that are computationally feasible in practice, such as the DNML. 

Empirical evaluations on both simulated and benchmark real-world networks indicate that the DNML estimator performs competitively with existing methods and often provides superior accuracy in unbalanced networks. However, the simulations also indicate that the DNML penalized estimator tends to underestimate the true number of communities, particularly when compared with the unpenalized DNML and the DNML with a logarithmic ($\log n$) penalty. In fact, our results guarantee that DNML and NMCL do not underestimate the true number of communities when the penalty function is of order $\log n$. However, proving non-overestimation for such a penalty, or for penalties of smaller order, remains challenging. Establishing theoretical consistency for these variants therefore remains an open problem and a direction for future research.

\section{Acknowledgments}
This research was supported by the São Paulo Research Foundation (FAPESP) grant 2023/05857-9 and 2025/08655-3 to A.C. and grant 2025/04228-3 to F.B.. This work was also supported by CNPq project (grant 441884/2023-7) and  by  São Paulo Research Foundation  (FAPESP) grant 2023/13453-5.

\bibliographystyle{chicago}

\bibliography{references}

\appendix

\newpage
\section{Proof of auxiliary results}

The key result for proving the consistency of the proposed estimators is to establish an upper bound for the normalizing constants $\CNML$, $\CNMCL$, $\CDMNL$ and $\CDMNLz$ that is of order $\log n$. Although the normalizing term $\CDMNL$ depends on the community assignments $\z$, the obtained upper bound is uniform over all possible community assignments. It is worth noting that this bound is closely related to a BIC-type penalty, which accounts for the number of free parameters multiplied by the logarithm of the sample size. Indeed, the probability vector $\pi$ involves $k-1$ free parameters associated with the $k$ communities of the $n$ nodes, while the symmetric connectivity matrix $P$ has $k(k+1)/2$ free parameters and generates a network of size $n^2$.

\begin{proposition}\label{prop:razao_L_KT}
For $k \geq 1$ and for all $\z \in [k]^n$, it holds that
\begin{equation}
    0 \leq \log\CDMNL(\z,k) + \log\CDMNLz(n,k)  \leq  \left(\frac{k(k+2)-1}{2} \right) \log n + c_k
\end{equation}
where $c_k$ is constant depending only on $k$.
Moreover, 
\begin{equation}
    0 \leq \log\CNML(n,k)  \leq \left(\frac{k(k+2)-1}{2} \right)\log n + d_k
\end{equation}
and
\begin{equation}
    0 \leq \log\CNMCL(n,k)  \leq \left(\frac{k(k+2)-1}{2} \right)\log n + d_k
\end{equation}
where $d_k$ is constant depending only on $k$.
\end{proposition}

\begin{proof}

For simplicity, we write $o_{ab}(\a,\z)=o_{ab}$, $n_{ab}(\z)=n_{ab}$ and $n_{a}(\z)=n_{a}$. Using that the conditional maximum likelihood estimator of $P_{ab}$ is given by  $o_{ab}/n_{ab}$, we write
\begin{equation}\label{eq:max_likelihood}
    \sup_{P \in \mathcal{P}^{k}} \P(\a\mid \z)  = \prod\limits_{a \leq b}  \left( \dfrac{o_{ab}}{n_{ab}} \right)^{o_{ab}}\left(1- \dfrac{o_{ab}}{n_{ab}} \right)^{n_{ab}-o_{ab}}.
\end{equation}
We also have that the conditional maximum likelihood estimator of $\pi_{a}$ is given by  $n_{a}/n$, and
\begin{equation}\label{eq:max_likelihood_z}
    \sup_{\pi \in \Pi^{k}} \Pp( \z)  = \prod\limits_{a=1}^k  \left( \frac{n_a}{n}\right)^{n_a}.
\end{equation}

To obtain an upper bound for the normalizing constants $\CNML$, $\CNMCL$, $\CDMNL$, and $\CDMNLz$, we follow the approach of Proposition 1 in \cite{cerqueira2020}. The central idea is to establish an upper bound for the ratio between the likelihood functions and their integrated forms, namely the marginal likelihoods obtained by integrating out the model parameters. 

We begin the analysis with the normalizing constant $\CDMNL$. In order to establish the result, the conditional maximum likelihood function will be compared with the integrated conditional likelihood. To define the integrated conditional likelihood, we define a prior distribution $\nu(P)$ over $[0,1]^{k\times k}$ by defining beta prior for $P_{ab}$ such that
$$P_{ab} \overset{\text{i.i.d.}}{\sim} \text{Beta}(1/2,1/2), \quad 1\leq a \leq b \leq k\,.$$

We define the integrated conditional likelihood, for $\z\in [k]^n$, by
\begin{equation}\label{eq:kt}
    \Q{\a\mid \z} = \mathbb E_{\nu}\left[\,\P(\a\mid \z)\,\right]
\end{equation}
where the expectation $E_{\nu}$ is with respect the distribution $\nu(P)$.
Using the conjugacy between the Bernoulli and Beta distributions, we have that
\begin{equation}\label{K_a_z}
\begin{split}
\Q{\a|\z}&=\int_{[0,1]^{\frac{k(k+1)}{2}}} \prod\limits_{a \leq b} \frac{1}{\Gamma(\frac{1}{2})^2} P_{ab}^{o_{ab}-\sfrac{1}{2}}(1-P_{ab})^{n_{ab}-o_{ab}-\sfrac{1}{2}} dP\\
&=\dfrac{1}{\Gamma\left(\frac{1}{2}\right)^{k(k+1)}}\prod\limits_{a \leq b} \dfrac{\Gamma\left(o_{ab}+\frac{1}{2}\right)\Gamma\left(n_{ab}-o_{ab}+\frac{1}{2}\right)}{\Gamma\left(n_{ab}+1\right)}\,.
\end{split}
\end{equation}

Combining \eqref{eq:max_likelihood} and \eqref{K_a_z}
\begin{equation}\label{razao_a|z}
\begin{split}
\dfrac{\sup_{P \in \mathcal{P}^{k}}\P(\a|\z)}{\Q{\a|\z}}  
= \Gamma\left(\dfrac{1}{2}\right)^{k(k+1)}\prod\limits_{a \leq b}  \dfrac{\left( \dfrac{o_{ab}}{n_{ab}} \right)^{o_{ab}}\left(1- \dfrac{o_{ab}}{n_{ab}} \right)^{n_{ab}-o_{ab}}}{\Gamma\left(o_{ab}+\frac{1}{2}\right)\Gamma\left(n_{ab}-o_{ab}+\frac{1}{2}\right)}\Gamma\left(n_{ab}+1\right)
\end{split}
\end{equation}

For each pair $a$ and $b$, $1\leq a \leq b \leq k$, we use Lemma 5.1 in \cite{Cerqueira24_DCSBM} to have that\\
\begin{equation}\label{eq:lemma_Ol}
 \dfrac{\left( \dfrac{o_{ab}}{n_{ab}} \right)^{o_{ab}}\left(1-\dfrac{o_{ab}}{n_{ab}} \right)^{n_{ab}-o_{ab}}}{\Gamma\left(o_{ab}+\frac{1}{2}\right)\Gamma\left(n_{ab}-o_{ab}+\frac{1}{2}\right)} \leq \dfrac{1}{\Gamma\left(n_{ab}+\frac{1}{2}\right)\, \Gamma\left( \frac{1}{2}\right)},
\end{equation}
using the fact that the sum of the exponents is equal to  $n_{ab}$.

Using  \eqref{eq:lemma_Ol} in \eqref{razao_a|z} we obtain
\begin{equation}\label{razao_final}
\begin{split}
\log \dfrac{\sup_{P \in \mathcal{P}^{k}}\P(\a\mid \z)}{\Q{\a\mid \z}} &\leq \frac{k(k+1)}{2}\log\Gamma(1/2) +  \sum\limits_{a\leq b}\log \left(  \dfrac{\Gamma\left(n_{ab}+1\right)}{\Gamma\left(n_{ab}+\frac{1}{2}\right) } \right)\,.
\end{split}
\end{equation}

In order to bound the sum on the right-hand side of \eqref{razao_final}, we use a bound that generalizes Stirling's approximation to the factorial \citep{whittaker2020course}[pag. 253]. For $x > 1$ we have that
\begin{equation}\label{eq:bound_gamma}
x^{x-\frac{1}{2}}e^{-x}\sqrt{2\pi} \leq \Gamma(x) \leq x^{x-\frac{1}{2}}e^{-x}\sqrt{2\pi}e^{\frac{1}{12x}}\,.
\end{equation}

Using \eqref{eq:bound_gamma} we have that
\begin{equation}\label{log_gamma_nlm}
\begin{split}
\log \left(  \dfrac{\Gamma\left(n_{ab}+1\right)}{\Gamma\left(n_{ab}+\frac{1}{2}\right) } \right) &\leq \frac{1}{2}\log n_{ab} +\frac{1}{2n_{ab}}+  \frac{1}{12 n_{ab}} + \frac{1}{2} \\
&\leq \log n + 2 \\
\end{split}
\end{equation}
where the last inequality follows from the fact that $1\leq n_{ab}(\z) \leq n^2$, for all $\z \in [k]^n$.

Combining \eqref{razao_final} and \eqref{log_gamma_nlm}, we have for all $\a \in \{0,1\}^{n \times n}$ and $\z \in [k]^n$ that
\begin{equation}
\begin{split}\label{eq:ratio_P_KT}
 \log \dfrac{\sup_{P \in \mathcal{P}^{k}}\P(\a\mid \z)}{\Q{\a\mid \z}} &\leq \frac{k(k+1)}{2} \log n + c'_k, 
\end{split}
\end{equation}
where $c'_k$ is a constant depending only on $k$.

Using the fact that $\Q{\a\mid \z}$ is a distribution over $\a\in \{0,1\}^{n\times n}$, for all $\z \in [k]^n$, we have that
\begin{equation}
\begin{split}
 \CDMNL(\z,k) &= \sum\limits_{\x}\sup\limits_{P \in \mathcal{P}^{k}}\P(\x\mid \z)\\
 &\leq \sum\limits_{\x}\Q{\x\mid \z}\max\limits_{\a}\dfrac{ \sup_{P \in \mathcal{P}^{k}} \P(\a\mid \z) }{\Q{\a\mid \z}}\\
 &\leq \max\limits_{\a}\dfrac{ \sup_{P \in \mathcal{P}^{k}} \P(\a\mid \z) }{\Q{\a\mid \z}}.
\end{split}
\end{equation}
Thus, the upper bound for $\log\CDMNL(\z,k)$ is given by
\begin{equation}
    \log\CDMNL(\z,k) \leq \max\limits_{\a}\log\dfrac{ \sup_{P \in \mathcal{P}^{k}} \P(\a\mid \z) }{\Q{\a\mid \z}} \leq \frac{k(k+1)}{2} \log n +  c'_k.
\end{equation}
To prove the lower bound for $\log\CDMNL(\z,k)$ is enough to observe that
\begin{equation}
\begin{split}
 \CDMNL(\z,k) &= \sum\limits_{\x}\sup\limits_{P \in \mathcal{P}^{k}}\P(\x\mid \z)\geq  \sup\limits_{P \in \mathcal{P}^{k}}\sum\limits_{\x}\P(\x\mid \z) =1.
\end{split}
\end{equation}

Thus, we conclude that
\begin{equation}\label{eq:bound_CDMNL}
    0 \leq \log  \CDMNL(\z,k) \leq \frac{k(k+1)}{2} \log n + c'_k.
\end{equation}

To prove the result for $\CDMNLz$, we proceed analogously, now introducing a prior distribution $\mu(\pi)$ on $[0,1]^k$, given by a Dirichlet$(1/2,1/2,\dots,1/2)$ prior for $\pi$. We then define the corresponding integrated likelihood by
\begin{equation}\label{eq:kt_z}
    \Q{ \z} = \mathbb E_{\mu}\left[\,\Pp(\z)\,\right].
\end{equation}
Using the conjugacy between the Dirichlet and Multinomial distributions, we have that
\begin{equation}\label{K_z}
\begin{split}
\Q{\z}&=\frac{\Gamma(\frac{k}{2})}{\Gamma(\frac{1}{2})^k}\frac{\prod\limits_{a=1}^k \Gamma(n_a+\frac{1}{2})}{ \Gamma(n+\frac{k}{2})}\,.
\end{split}
\end{equation}
Combining \eqref{eq:max_likelihood_z} and \eqref{K_z} we obtain
\begin{equation}\label{razao_z}
\begin{split}
\dfrac{\sup_{\pi \in \Pi^{k}}\Pp(\z)}{\Q{\z}}  
=  \frac{\Gamma(n+\frac{k}{2})\Gamma(\frac{1}{2})^k}{\Gamma(\frac{k}{2})}\prod_{a=1}^k \frac{\left( \frac{n_a}{n}\right)^{n_a}}{\Gamma(n_a+\frac{1}{2})}.
\end{split}
\end{equation}
Using the fact that $n_1 + n_2 + \dots + n_k = n$, we apply Lemma~5.1 in \cite{Cerqueira24_DCSBM} once again to obtain that
\begin{equation}
\prod_{a=1}^k  \frac{\left( \frac{n_a}{n}\right)^{n_a}}{\Gamma(n_a+\frac{1}{2})}\leq \frac{1}{\Gamma(\frac{1}{2})^{k-1}\Gamma(n+\frac{1}{2})}.
\end{equation}
Thus, we conclude that
\begin{equation}\label{eq:bound_gamas_z}
\dfrac{\sup_{\pi \in \Pi^{k}}\Pp(\z)}{\Q{\z}}  
\leq \frac{\Gamma(\frac{1}{2})\Gamma(n+\frac{k}{2})}{\Gamma(\frac{k}{2})\Gamma(n+\frac{1}{2})}.
\end{equation}
By \eqref{eq:bound_gamma}, we obtain that
\begin{equation}\label{eq:bound_KTz_z}
\log \left( \frac{\Gamma(\frac{1}{2})\Gamma(n+\frac{k}{2})}{\Gamma(\frac{k}{2})\Gamma(n+\frac{1}{2})} \right) \leq \left( \frac{k-1}{2}\right)\log n + c''_k,
\end{equation}
where $c''_k$ is a constant depending only on $k$.
Combining \eqref{eq:bound_gamas_z} and \eqref{eq:bound_KTz_z} we have, for all $\z\in [k]^n$, that
\begin{equation}\label{eq:ratio_P_KT_z}
\log\left(\dfrac{\sup_{\pi \in \Pi^{k}}\Pp(\z)}{\Q{\z}} \right)\leq \left( \frac{k-1}{2}\right)\log n + c''_k 
\end{equation}
Using the fact that $\Q{\z}$ is a distribution over $\z\in [k]^n$, we have that
\begin{equation}
\begin{split}
 \CDMNLz(n,k) &= \sum\limits_{\e}\sup\limits_{\pi \in \Pi^{k}}\Pp(\e) \leq \sum\limits_{\e}\Q{\e}\max\limits_{\z}\dfrac{\sup\limits_{\pi \in \Pi^{k}}\Pp(\z)}{\Q{\z}} \leq \max\limits_{\z}\dfrac{\sup\limits_{\pi \in \Pi^{k}}\Pp(\z)}{\Q{\z}}
\end{split}
\end{equation}
and, by taking the logarithm, we conclude that
\begin{equation}
\begin{split}
 \log\CDMNLz(n,k) \leq \left( \frac{k-1}{2}\right)\log n + c''_k.
\end{split}
\end{equation}
On the other hand, we have that
\begin{equation}
\begin{split}
 \CDMNLz(n,k) &=  \sum\limits_{\e}\sup\limits_{\pi \in \Pi^{k}}\Pp(\e) \geq \sup\limits_{\pi \in \Pi^{k}}\sum\limits_{\e}\Pp(\e) =1.
\end{split}
\end{equation}
Thus, we conclude that 
\begin{equation}\label{eq:bound_CDMNLz}
    0 \leq \log\CDMNLz(n,k) \leq \left( \frac{k-1}{2}\right)\log n + c''_k.
\end{equation}

Consequently, the result for $\log\CDMNL+\log\CDMNLz$ follows directly from the sum of \eqref{eq:bound_CDMNL} and \eqref{eq:bound_CDMNLz}.

To prove the result for $\CNML$, we follow the same steps as before. 
However, instead of comparing the conditional maximum likelihood with the integrated conditional likelihood, we compare the likelihood with the integrated completed likelihood, defined as
\begin{equation}\label{eq:kt2}
    \Q{\a} = \mathbb{E}_{\nu,\mu}\!\left[\,\Pt(\a)\,\right].
\end{equation}
First, observe that 
\begin{equation}
\Q{\a} = \sum\limits_{\z}\Q{\a\mid \z}\Q{\z}.
\end{equation}
Thus, we have, for all $\a\in\{0,1\}^{n\times n}$, that
\begin{equation}\label{eq:bound_kt}
\begin{split}
\log \left(\sup\limits_{\pi \in  \Pi^{k}, P\in\mathcal{P}^{k} }\Pt(\a)\right) &= \log \left(\sup\limits_{\pi \in  \Pi^{k}, P\in\mathcal{P}^{k} }\sum\limits_{\z}\Pt(\a\mid \z)\Pp(\z)\right)\\
& \leq \log \left(\sup\limits_{\pi \in  \Pi^{k}, P\in\mathcal{P}^{k} }\sum\limits_{\z}n^{\frac{k(k+2)-1}{2}}e^{c'_k+c''_k}\Q{\a\mid \z}\Q{\z}\right)\\
& = \log\left(\sup\limits_{\pi \in  \Pi^{k}, P\in\mathcal{P}^{k} }n^{\frac{k(k+2)-1}{2}}e^{c'_k+c''_k}\Q{\a}\right)\\
& = \log \Q{\a} + \left(\frac{k(k+2)-1}{2}\right)\log n + d_k
\end{split}
\end{equation}
where $d_k$ is a constant depending on $k$.
Using that $\Q{\a}$ is a distribution over $\a\in \{0,1\}^{n\times n}$ we have
\begin{equation}\label{eq:bound_kt2}
\begin{split}
 \CNML(n,k) &= \sum\limits_{\x}\sup\limits_{\pi \in  \Pi^{k}, P\in\mathcal{P}^{k} }\Pt(\x)\\
 &\leq \sum\limits_{\x}\Q{\x}\max\limits_{\a}\dfrac{ \sup_{\pi \in  \Pi^{k}, P\in\mathcal{P}^{k} } \Pt(\a) }{\Q{\a}}\\
 &\leq \max\limits_{\a}\dfrac{ \sup_{\pi \in  \Pi^{k}, P\in\mathcal{P}^{k} } \Pt(\a) }{\Q{\a}}.
\end{split}
\end{equation}
The upper bound for $\CNML(n,k)$ is obtained by combining \eqref{eq:bound_kt} and \eqref{eq:bound_kt2}. The lower bound is obtained by observing that
\begin{equation}
\begin{split}
 \CNML(n,k) &= \sum\limits_{\x}\sup\limits_{\pi \in  \Pi^{k}, P\in\mathcal{P}^{k} }\Pt(\x)\geq  \sup\limits_{\pi \in  \Pi^{k}, P\in\mathcal{P}^{k} }\sum\limits_{\x}\Pt(\x) =1.
\end{split}
\end{equation}
Finally, for $\CNMCL$ we write
\begin{equation}
\begin{split}
 \CNMCL(n,k) &= \sum_{\e}\sum\limits_{\x}\sup\limits_{\pi \in  \Pi^{k}, P\in\mathcal{P}^{k} }\Pt(\x,\e)\\
 &\leq \sum_{\e}\sum\limits_{\x}\Q{\x\mid \e}\Q{\e}\max_{\a,\z}\dfrac{ \sup_{\pi \in  \Pi^{k}, P\in\mathcal{P}^{k} } \Pt(\a,\z) }{\Q{\a\mid \z}\Q{\z}}\\
 &\leq \max\limits_{\a,\z}\dfrac{ \sup_{\pi \in  \Pi^{k}, P\in\mathcal{P}^{k} } \Pt(\a,\z) }{\Q{\a\mid \z}\Q{\z}}.
\end{split}
\end{equation}
Combining \eqref{eq:ratio_P_KT} and \eqref{eq:ratio_P_KT_z} the result follows.
\end{proof}

Propostions \ref{prop:bound_prob_overestimation_DNML} and \ref{prop:bound_prob_overestimation_NML} are useful for showing that the estimator does not overestimate the true number of communities, as it provides an upper bound on the probability that the estimator selects $k$ communities when $k > k_0$. In Proposition \ref{prop:bound_prob_overestimation_DNML}, an additional term of order $n\log k_0$ appears in the bound due to the use of the complete likelihood evaluated at the estimated community assignments. This extra term plays an important role in the form of the penalty function \eqref{eq:pen_DNML}. This term does not appear in Proposition \ref{prop:bound_prob_overestimation_NML}, as the estimator relies on the marginal likelihood rather than the likelihood evaluated at a specific community configuration.

\begin{proposition}\label{prop:bound_prob_overestimation_DNML}
 Suppose that $\A$ is generated from a SBM with $k_0$ communities and parameters $\pi_0$ and $P_0$. Let $\hat k_n$ denote any of the estimators $\hk$ or $\hkNMCL$. For $k > k_0$, it follows that
\begin{equation}
\Pto(\hat k_n(\A)=k) \leq \exp\left( \frac{k_0(k_0+2)-1}{2}\log n + n \log k_0 + d_{k_0} + \pen(k_0,n) -\pen(k,n) \right) .
\end{equation}
where $d_{k_0}$ depends only on $k_0$.
\end{proposition}

\begin{proof}
We begin by considering the estimator $\hk$. For any network $\a\in \{0,1\}^{n\times n}$, let $\hz$ be the communities estimated with $k_0$ communities, that is
\begin{equation}
    \hz(\a) =  \arg\max\limits_{\z \in [k_0]^n} \left\lbrace\sup\limits_{{\pi \in  \Pi^{k_0}, P\in\mathcal{P}^{k_0}}}\Pt(\a, \z) \right\rbrace.
\end{equation}

For $k>k_0$, define $\hz^*$ the communities estimated in the larger model given by
\begin{equation}
    \hz^{*}(\a) = \arg\max\limits_{\z \in [k]^n} \left\lbrace\sup\limits_{{\pi \in  \Pi^{k}, P\in\mathcal{P}^{k}}}\Pt(\a, \z) \right\rbrace.
\end{equation}

By Proposition \ref{prop:razao_L_KT}, for all $\a\in \{0,1\}^{n\times n}$ we have that
\begin{equation}\label{eq:MV_MCML}
\begin{split}
    \log\sup\limits_{{\pi \in  \Pi^{k_0}, P\in\mathcal{P}^{k_0}}}&\Pt(\a, \hz(\a))  - \log\DNML_{k_0}(\a, \hz(\a))  \\
     & = \log \CDMNL(\hz(\a),k_0)+\CDMNLz(n,k_0) \\
     &\leq \left(\frac{k_0(k_0+2)-1}{2} \right)\log n +c_{k_0}.
\end{split}
\end{equation}

Taking $k> k_0$, by the definition of the estimators \eqref{eq:est_NML} and \eqref{eq:est_DNML}, observe that the event $\{\hk(\A)=k\}$ is contained in the event 
\begin{equation}
\begin{split}
\mathcal{B}(\A&, \hz^*(\A),\hz(\A)) \\
& =\left\lbrace\log\DNML_k(\A, \hz^*(\A)) - \log\DNML_{k_0}(\A, \hz(\A))  >  -\Delta_{k_0,k}\right\rbrace\\
& = \{\DNML_{k_0}(\A,\hz(\A)) \leq\DNML_k(\A, \hz^*(\A)e^{\Delta_{k_0,k})}\}
\end{split}
\end{equation}
with
\[\Delta_{k_0,k}=\pen(k_0,n) -\pen(k,n).\]
It follows that
\begin{equation}\label{eq:ML_NCML}
\begin{split}
\quad \Pto(\hk(\A)&=k)  = \mathbb{E}_{\pi_0,P_0}[\ind\{ \hk(\A)=k \}] \\
&\leq\sum\limits_{\a}\Pto(\A=\a)\ind\{\mathcal{B}(\a, \hz^*(\a),\hz(\a))\}\\
&=\sum\limits_{\a}\sum\limits_{\z\in [k_0]^n}\Po(\a\mid \z)\mathbb P_{\pi_0}(\z)\ind\{\mathcal{B}(\a, \hz^*(\a),\hz(\a))\}.
\end{split}
\end{equation}
By Equation \eqref{eq:MV_MCML} and by the definition of $\hz(\a)$, we have
\begin{equation}
\begin{split}
\log\Po(\a\mid \z)P_{\pi_0}(\z) &\leq \log\sup\limits_{\pi \in  \Pi^{k_0}, P\in\mathcal{P}^{k_0}}\P(\a\mid \hz(\a))\Pp(\hz(\a))\\
&\leq \log \DNML_{k_0}(\a, \hz(\a)) + \left(\frac{k_0(k_0+2)-1}{2} \right)\log n +c_{k_0}.
\end{split}
\end{equation}
Combining the last equation with \eqref{eq:ML_NCML} and with the definition of $\mathcal{B}$, we get
\begin{equation}
\begin{split}
&\Pto(\hk(\A)=k)  \\
&\leq\sum\limits_{\a}\sum\limits_{\z\in [k_0]^n}\DNML_{k_0}(\a, \hz(\a))n^{\frac{k_0(k_0+2)-1}{2}}e^{c_{k_0}}\ind\{\mathcal{B}(\a, \hz^*(\a),\hz(\a))\}\\
&\leq \sum\limits_{\a}n^{\frac{k_0(k_0+2)-1}{2}}k_0^n e^{c_{k_0}}\DNML_k(\a, \hz^*(\a))e^{ \Delta_{k_0,k}}.
\end{split}
\end{equation}
The result follows for the estimator $\hk$ by using that $\sum\limits_{\a}\DNML_k(\a, \hz^*(\a))\leq 1$. Indeed,
\begin{equation}
\begin{split}
 \sum\limits_{\a}&\DNML_k(\a\mid \hz^*(\a))\DNML_k(\hz^*(\a)) \\
 & \leq \sum\limits_{\e\in [k]^n}\sum\limits_{\a}\DNML_k(\a\mid \e)\DNML_k(\e)\\
 & = 1
\end{split}
\end{equation}

We proceeded in the same way for the estimator $\hkNMCL$. First, by Proposition \ref{prop:razao_L_KT}, for all $\a\in \{0,1\}^{n\times n}$, we have that
\begin{equation}
\begin{split}
    \log\sup\limits_{{\pi \in  \Pi^{k_0}, P\in\mathcal{P}^{k_0}}}&\Pt(\a, \hz(\a))  - \log\NMCL_{k_0}(\a, \hz(\a))  \\
     & = \log \CNMCL(n,k_0) \\
     &\leq \left(\frac{k_0(k_0+2)-1}{2} \right)\log n +d_{k_0}.
\end{split}
\end{equation}
By defining the event
\begin{equation}
\begin{split}
\mathcal{E}(\A&, \hz^*(\A),\hz(\A)) \\
& = \{\NMCL_{k_0}(\A,\hz(\A)) \leq\NMCL_k(\A, \hz^*(\A))e^{\Delta_{k_0,k}}\}
\end{split}
\end{equation}
we obtain
\begin{equation}
\begin{split}
&\Pto(\hkNMCL(\A)=k)  \\
&\leq\sum\limits_{\a}\NMCL_{k_0}(\a, \hz(\a))n^{\frac{k_0(k_0+2)-1}{2}}k_0^n e^{d_{k_0}}\ind\{\mathcal{E}(\a, \hz^*(\a),\hz(\a))\}\\
&\leq \sum\limits_{\a}n^{\frac{k_0(k_0+2)-1}{2}}k_0^n e^{d_{k_0}}\NMCL_k(\a, \hz^*(\a))e^{ \Delta_{k_0,k}}.
\end{split}
\end{equation}
The result follows by observing that
\begin{equation}
\begin{split}
 \sum\limits_{\a}&\NMCL_k(\a, \hz^*(\a)) \leq  \sum\limits_{\e}\sum\limits_{\a}\NMCL_k(\a, \e) =1.
\end{split}
\end{equation}
\end{proof}

\begin{proposition}\label{prop:bound_prob_overestimation_NML}
 Suppose that $\A$ is generated from a SBM with $k_0$ communities and parameters $\pi_0$ and $P_0$. For $k > k_0$, it follows that
\begin{equation}
\Pto(\hkNML(\A)=k) \leq \exp\left( \frac{k_0(k_0+2)-1}{2}\log n + d_{k_0} + \pen(k_0,n) -\pen(k,n) \right) .
\end{equation}
where $d_{k_0}$ depends only on $k_0$.
\end{proposition}

\begin{proof}
For the estimator $\hkNML$, we observe that, for all $\a\in \{0,1\}^{n\times n}$, by Proposition \ref{prop:razao_L_KT} we have that
\begin{equation}\label{eq:MV_NML}
\begin{split}
    \log\sup\limits_{\pi \in  \Pi^{k_0}, P\in\mathcal{P}^{k_0}}\Pt(\a)  - \log\NML_{k_0}(\a)  &= \log \CNML(n,k_0) \\
     &\leq \frac{k_0(k_0+2)-1}{2} \log n +d_{k_0}.
\end{split}
\end{equation}
Thus, the result follows from the same argument used in the proof of Proposition \ref{prop:bound_prob_overestimation_DNML} by defining the event
\begin{equation}
\begin{split}
\mathcal{F}(\A) = \{\NML_{k_0}(\A) \leq\NML_k(\A)e^{\Delta_{k_0,k}}\}
\end{split}
\end{equation}
with $\Delta_{k_0,k}=\pen(k_0,n) -\pen(k,n)$ to obtain
\begin{equation}
\begin{split}
\Pto(\hkNML(\A)=k) &\leq\sum\limits_{\a}\NML_{k_0}(\a)n^{\frac{k_0(k_0+2)-1}{2}}e^{d_{k_0}}\ind\{\mathcal{F}(\a)\}\\
&\leq \sum\limits_{\a}n^{\frac{k_0(k_0+2)-1}{2}}e^{d_{k_0}}\NML_k(\a)e^{ \Delta_{k_0,k}}.
\end{split}
\end{equation}
We conclude the proof using the fact that $\sum_{\a}\NML_k(\a)=1$.    
\end{proof}

The next lemma is a technical result used to prove the non-underestimation property of the proposed estimators. It shows that, asymptotically, the likelihood of a SBM with the true number of communities dominates that of any model with fewer communities. This result was established in Lemma 7 (for $\rho_n=1$) and Lemma 8 (for $\rho_n\to 0$) in \cite{cerqueira2020}, and its proof is therefore omitted.

\begin{lemma}\label{lemma:cerqueira20}
    Suppose that $(\A,\Z)$ is generated from a SBM with $k_0$ communities and parameters $\pi_0$ and $P_0=\rho_nS_0$ with either $\rho_n \to 0$ and $n\rho_n\to \infty$, or $\rho_n = 1$. Then, we have that
    \begin{equation}
\begin{split}
\liminf_{n\to \infty}\frac{1}{\rho_n n^2}&\log  \frac{\sup\limits_{{\pi \in  \Pi^{k_0}, P\in\mathcal{P}^{k_0}}}\Pt(\A, \Z)}{\sup\limits_{{\pi \in  \Pi^{k_0-1}, P\in\mathcal{P}^{k_0-1}}}\Pt(\A, \hz^*(\A))} > 0
\end{split}
\end{equation}
with $\hz^{*}(\a) = \arg\max\limits_{\z \in [k_0-1]^n} \left\lbrace \sup\limits_{{\pi \in  \Pi^{k_0-1}, P\in\mathcal{P}^{k_0-1}}}\Pt(\a, \z)\right\rbrace$.
\end{lemma}

\section{Non-overestimation}\label{sec:non_overestimation}

In this section, we prove that none of the estimators $\hkNML$, $\hkNMCL$, or $\hk$ overestimate the true number of communities $k_0$. The proof is divided into two parts: first, we show that the estimator does not fall within the interval $(k_0, \log n]$, and second, that it does not fall within the interval $(\log n, n]$, eventually almost surely as $n\to \infty$.
The key idea of the following result is to choose an appropriate expression for the penalty term $\pen(k,n)$ so that the probability of selecting $k$ communities, as given in Propositions \ref{prop:bound_prob_overestimation_DNML} and \ref{prop:bound_prob_overestimation_NML}, is summable over $k > k_0$ and $n$. Consequently, the specific formula for the penalty function for each estimator follows directly from these results.

\begin{proposition}\label{prop:non_overestimation_DNML}
Suppose that $\A$ is generated from a SBM with $k_0$ communities and parameters $\pi_0$ and $P_0$. Then, for the penalty function given by \eqref{eq:pen_DNML}, the estimators $\hkNMCL(\A)$ and $\hk(\A)$ do not overestimate $k_0$ eventually almost surely as $n \to \infty$.
\end{proposition}

\begin{proof}
Let $\hat k_n$ denote any of the estimators $\hkNMCL$ or $\hk$.
We first show that \[\hat k_n(\A)\notin (k_0,\log n] \]
eventually almost surely as $n \to \infty$. By Proposition \ref{prop:bound_prob_overestimation_DNML}, we obtain
\begin{equation}
 \begin{split}
\Pto(\hat k_n&\in (k_0,\log n]) = \sum\limits_{k=k_0+1}^{\log n} \Pto(\hat k_n(\A)=k)\\
& \leq \sum\limits_{k=k_0+1}^{\log n} \exp\left( \frac{k_0(k_0+2)-1}{2}\log n + n\log k_0+ d_{k_0} + \pen(k_0,n) -\pen(k,n) \right) \\
& \leq \log n \exp\left( \frac{k_0(k_0+2)-1}{2}\log n + n\log k_0 + d_{k_0} + \pen(k_0,n) -\pen(k_0+1,n) \right) .
 \end{split}
\end{equation}
Using the penalty function defined in \eqref{eq:pen_DNML} we have, for $\epsilon>0$, that
\begin{equation}
\pen(k_0,n) -\pen(k_0+1,n) = - \left( \frac{k_0(k_0+2)-1}{2} + 1+ \epsilon/2\right)\log n -n\log k_0.
\end{equation}
The result follows from the Borel–Cantelli lemma, since
\begin{equation}
    \sum\limits_{n=1}^{\infty}\Pto(\hkNML(\A) \in (k_0,\log n]) =  e ^{d_{k_0}}\sum\limits_{n=1}^{\infty}\frac{\log n}{n^{1+\epsilon/2}} < \infty.
\end{equation}
We now show that \[\hat k_n(\A)\notin (k_0,\log n] \]
eventually almost surely as $n \to \infty$.  Following the same approach as before, we use Proposition \ref{prop:bound_prob_overestimation_DNML} to obtain
\begin{equation}
 \begin{split}
\Pto(\hat k_n(\A)&\in (\log n,n]) \leq \sum\limits_{k=\log n}^{n} \Pto(\hat k_n(\A)=k)\\
& \leq n \exp\left( \frac{k_0(k_0+2)-1}{2}\log n + n\log(k_0)+d_{k_0} + \pen(k_0,n) -\pen(\log n,n) \right) \\
& \leq e^{d_{k_0}} n \exp\left[ -b_n\log n \right] .
 \end{split}
\end{equation}
with $b_n=\left(-\frac{k_0(k_0+2)-1}{2} -\frac{\pen(k_0,n)}{\log n} +\frac{\pen(\log n,n)}{\log n} -  \frac{n\log k_0}{\log n}\right)$. Since $\pen(k,n)$ contains a term of order $n \log\big((k-1)!\big)$, the dominant contribution to $b_n$ comes from $\pen(\log n,n)/\log n$, and hence $b_n \to \infty$ as $n \to \infty$.
Thus, the result follows from the Borel–Cantelli lemma, since
\begin{equation}
    \sum\limits_{n=1}^{\infty}\Pto(\hat k_n(\A) \in (\log n,n])  < \infty.
\end{equation}
\end{proof}

\begin{proposition}\label{prop:non_overestimation_NML}
Suppose that $\A$ is generated from a SBM with $k_0$ communities and parameters $\pi_0$ and $P_0$. Then, for the penalty function given by \eqref{eq:pen_NML}, the estimators $\hkNML(\A)$ do not overestimate $k_0$ eventually almost surely as $n \to \infty$.
\end{proposition}

\begin{proof}
We first show that \[\hkNML(\A)\notin (k_0,\log n] \]
eventually almost surely as $n \to \infty$. By Proposition \ref{prop:bound_prob_overestimation_NML}, we obtain
\begin{equation}
 \begin{split}
\Pto(\hkNML &\in (k_0,\log n]) = \sum\limits_{k=k_0+1}^{\log n} \Pto(\hkNML(\A)=k)\\
& \leq \sum\limits_{k=k_0+1}^{\log n} \exp\left( \frac{k_0(k_0+2)-1}{2}\log n + d_{k_0} + \pen(k_0,n) -\pen(k,n) \right) \\
& \leq \log n \exp\left( \frac{k_0(k_0+2)-1}{2}\log n + d_{k_0} + \pen(k_0,n) -\pen(k_0+1,n) \right) .
 \end{split}
\end{equation}
Using the penalty function defined in \eqref{eq:pen_NML} we have, for $\epsilon>0$, that
\begin{equation}
\pen(k_0,n) -\pen(k_0+1,n) = - \left( \frac{k_0(k_0+2)-1}{2} + 1+ \epsilon/2\right)\log n.
\end{equation}
The result follows from the Borel–Cantelli lemma, since
\begin{equation}
    \sum\limits_{n=1}^{\infty}\Pto(\hkNML(\A) \in (k_0,\log n]) =  e ^{c_{k_0}}\sum\limits_{n=1}^{\infty}\frac{\log n}{n^{1+\epsilon/2}} < \infty.
\end{equation}
We now show that \[\hkNML(\A)\notin (k_0,\log n] \]
eventually almost surely as $n \to \infty$.  Following the same approach as before, we use Proposition \ref{prop:bound_prob_overestimation_NML} to obtain
\begin{equation}
 \begin{split}
\Pto(\hkNML(\A)&\in (\log n,n]) \leq \sum\limits_{k=\log n}^{n} \Pto(\hkNML(\A)=k)\\
& \leq n \exp\left( \frac{k_0(k_0+2)-1}{2}\log n + c_{k_0} + \pen(k_0,n) -\pen(\log n,n) \right) \\
& \leq e^{c_{k_0}} n \exp\left[ -\log n\left(-\frac{k_0(k_0+2)-1}{2} -\frac{\pen(k_0,n)}{\log n} +\frac{\pen(\log n,n)}{\log n}\right) \right] .
 \end{split}
\end{equation}
Since $\pen(k,n)/\log n$ grows cubically with respect to $k$, it follows, for $n$ sufficiently large, that
\begin{equation}
-\frac{k_0(k_0+2)-1}{2} -\frac{\pen(k_0,n)}{\log n} +\frac{\pen(\log n,n)}{\log n} > 3.
\end{equation}
Thus, the result follows from the Borel–Cantelli lemma, since
\begin{equation}
    \sum\limits_{n=1}^{\infty}\Pto(\hkNML(\A) \in (\log n,n])  < \infty.
\end{equation}
\end{proof}

\section{Non-underestimation}

In this section we show that the estimators $\hkNML$, $\hkNMCL$, and $\hk$ recover at least the true number of communities $k_0$, so that underestimation does not occur, eventually almost surely as $n \to \infty$. In this context, network sparsity plays an important role, as extremely sparse networks provide limited information to distinguish communities, which could otherwise lead to underestimation. In order to guarantee the consistency of the estimator, the penalty function needs to satisfy, for $k < k_0$,
\begin{equation}\label{eq:penalty_rho}
    \lim_{n\to \infty} \frac{\pen(k_0,n)-\pen(k,n)}{\rho_n n^2} = 0.
\end{equation}
Observe that both penalty functions defined in \eqref{eq:pen_NML} and \eqref{eq:pen_DNML} satisfy this condition provided that $\rho_n \to 0$ and $n\rho_n \to \infty$ as $n \to \infty$.

The next result guarantees that underestimation does not occur in the sparse regime, where $\rho_n\to 0$ as $n\rho_n \to \infty$. In fact, note that the non-underestimation of the estimators $\hk$ and $\hkNMCL$ is proved for a penalty function of order $\log n$. Thus, the additional term required in the penalty function in \eqref{eq:pen_DNML} arises from the proof of non-overestimation.

\begin{proposition}\label{prop:non_underestimation}
     Suppose that $\A$ is generated from a SBM with $k_0$ communities and parameters $\pi_0$ and $P_0=\rho_nS_0$ with either $\rho_n \to 0$ and $n\rho_n\to \infty$, or $\rho_n = 1$. If the penalty function $\pen(k,n)$ satisfies \eqref{eq:penalty_rho}, for $k\leq k_0$, then the estimators $\hkNML(\A)$, $\hkNMCL(\A)$ and $\hk(\A)$ do not underestimate $k_0$, eventually almost surely when $n\to \infty$.
\end{proposition}

\begin{proof}
We begin by considering the estimator $\hk$. Let $\hz$ be the communities estimated with $k_0$ communities. For $k <k_0$, define $\hz^*$ the communities estimated in the smaller model, that is,
\begin{equation}
    \hz^{*}(\a) = \arg\max\limits_{\z \in [k]^n} \left\lbrace \sup\limits_{{\pi \in  \Pi^{k}, P\in\mathcal{P}^{k}}}\Pt(\a, \z)\right\rbrace.
\end{equation}
To prove the result it is enough to prove that for all $k < k_0$,
\begin{equation}
    \log\DNML_{k_0}(\A\, \hz(\A) ) - \pen(k,n) \geq \log\DNML_{k}(\A,\hz^*(\A) ) - \pen(k,n)
\end{equation}
eventually almost surely when $n\to \infty$.

For $\rho_n\to 0$ as $n\rho_n\to\infty$, using the assumption that
\begin{equation}
    \lim_{n\to \infty} \frac{\pen(k_0,n)-\pen(k,n)}{\rho_n n^2} =0
\end{equation}
it is enough to show that, eventually almost surely when $n\to \infty$,
\begin{equation}
    \liminf_{n\to \infty}\frac{1}{\rho_n n^2}\log \frac{\DNML_{k_0}(\A, \hz(\A) )}{\DNML_{k}(\A,\hz^*(\A) )} > 0.
\end{equation}

Using the definition of DNML we write
\begin{equation}
\begin{split}
\frac{1}{\rho_n n^2}\log \frac{\DNML_{k_0}(\A, \hz(\A) )}{\DNML_{k}(\A, \hz^*(\A) )} &=  \frac{1}{\rho_n n^2}\log \frac{\sup\limits_{{\pi \in  \Pi^{k_0}, P\in\mathcal{P}^{k_0}}}\Pt(\A, \hz(\A))}{\sup\limits_{{\pi \in  \Pi^{k}, P\in\mathcal{P}^{k}}}\Pt(\A, \hz^*(\A))} \\
& -\frac{1}{\rho_n n^2}\log\tilde C(\hz(\A),k_0) + \frac{1}{\rho_n n^2}\log\tilde C(\hz^*(\A),k),
\end{split}
\end{equation}
where
\begin{equation}
\log \tilde C(\e,k) = \log\CDMNL(\e,k) + \log\CDMNLz(n,k).
\end{equation}
By the uniform bound obtained in Proposition \ref{prop:razao_L_KT} we have that
\begin{equation}
    \begin{split}
         \frac{1}{\rho_n n^2}\log\tilde C(\hz^*(\A),k) -\frac{1}{\rho_n n^2}\log \tilde C(\hz(\A),k_0) &\geq  -\frac{1}{\rho_n n^2}\log\tilde C(\hz(\A),k_0) \\
         &\geq  -\frac{(k_0(k_0+2)-1)\log n-2c_{k_0}}{2\rho_n n^2}
    \end{split}
\end{equation}
goes to zero as $n\to \infty$ and $n\rho_n\to \infty$. Thus, it is enough to show that, eventually almost surely when $n\to \infty$,
\begin{equation}
      \liminf_{n\to \infty}\frac{1}{\rho_n n^2}\log \frac{\sup\limits_{{\pi \in  \Pi^{k_0}, P\in\mathcal{P}^{k_0}}}\Pt(\A, \hz(\A))}{\sup\limits_{{\pi \in  \Pi^{k}, P\in\mathcal{P}^{k}}}\Pt(\A, \hz^*(\A))} > 0.
\end{equation}

Without loss of generality, we begin by setting $k=k_0-1$.
Let the pair $(\Z,\A)$ be generated from the SBM with $k_0$ communities and parameters $\pi_0$ and $P_0$. We write
\begin{equation}
    \sup\limits_{{\pi \in  \Pi^{k_0}, P\in\mathcal{P}^{k_0}}}\Pt(\A, \hz(\A))  \geq \sup\limits_{{\pi \in  \Pi^{k_0}, P\in\mathcal{P}^{k_0}}}\Pt(\A, \Z)
\end{equation}
Thus, we have that 
\begin{equation}
\begin{split}
\liminf_{n\to \infty}\frac{1}{\rho_n n^2}&\log  \frac{\sup\limits_{{\pi \in  \Pi^{k_0}, P\in\mathcal{P}^{k_0}}}\Pt(\A, \hz(\A))}{\sup\limits_{{\pi \in  \Pi^{k_0-1}, P\in\mathcal{P}^{k_0-1}}}\Pt(\A, \hz^*(\A)} \\
&\geq \liminf_{n\to \infty}\frac{1}{\rho_n n^2}  \frac{\sup\limits_{{\pi \in  \Pi^{k_0}, P\in\mathcal{P}^{k_0}}}\Pt(\A, \Z)}{\sup\limits_{{\pi \in  \Pi^{k_0-1}, P\in\mathcal{P}^{k-1}}}\Pt(\A, \hz^*(\A))}\\
     &> 0
\end{split}
\end{equation}
where the last inequality follows from Lemma \ref{lemma:cerqueira20}.

To complete the proof for $k' < k_0-1$, let $\hz^{**}(\A)$ be the estimated communities for  a smaller model with $k'$ communities, we write
\begin{equation}
     \begin{split}
         &\log \frac{\sup\limits_{{\pi \in  \Pi^{k_0}, P\in\mathcal{P}^{k_0}}}\Pt(\A, \hz(\A))}{\sup\limits_{{\pi \in  \Pi^{k'}, P\in\mathcal{P}^{k'}}}\Pt(\A, \hz^{**}(\A))}\\
         &= \log \frac{\sup\limits_{{\pi \in  \Pi^{k_0}, P\in\mathcal{P}^{k_0}}}\Pt(\A, \hz(\A))}{\sup\limits_{{\pi \in  \Pi^{k_0-1}, P\in\mathcal{P}^{k_0-1}}}\Pt(\A, \hz^*(\A))} + \log \frac{\sup\limits_{{\pi \in  \Pi^{k_0-1}, P\in\mathcal{P}^{k_0-1}}}\Pt(\A, \hz^*(\A))}{\sup\limits_{{\pi \in  \Pi^{k'}, P\in\mathcal{P}^{k'}}}\Pt(\A, \hz^{**}(\A))}\\
         & \geq \log \frac{\sup\limits_{{\pi \in  \Pi^{k_0}, P\in\mathcal{P}^{k_0}}}\Pt(\A, \hz(\A))}{\sup\limits_{{\pi \in  \Pi^{k_0-1}, P\in\mathcal{P}^{k_0-1}}}\Pt(\A, \hz^*(\A))}
     \end{split}
\end{equation}
 where the first inequality follows from the fact that $\Pi^{k'}\subset \Pi^{k_0-1}$ and $\mathcal{P}^{k'} \subset \mathcal{P}^{k_0-1}$, $k'< k_0$, implying that the maximum likelihood function is a
non-decreasing function of the dimension of the model, that is,
\begin{equation}
\begin{split}
\sup\limits_{{\pi \in  \Pi^{k_0-1}, P\in\mathcal{P}^{k_0-1}}}\Pt(\A, \hz(\A)) &= \max_{\z\in [k_0-1]^n}\sup\limits_{{\pi \in  \Pi^{k_0-1}, P\in\mathcal{P}^{k_0-1}}}\Pt(\A, \z)\\
& \geq \max_{\z\in [k_0-1]^n}\sup\limits_{{\pi \in  \Pi^{k'}, P\in\mathcal{P}^{k'}}}\Pt(\A, \z)\\
&\geq \max_{\z\in [k']^n}\sup\limits_{{\pi \in  \Pi^{k'}, P\in\mathcal{P}^{k'}}}\Pt(\A, \z).
\end{split}
\end{equation}
Thus, we conclude that
\begin{equation}
\liminf_{n\to \infty}\frac{1}{\rho_n n^2}\log  \frac{\sup\limits_{{\pi \in  \Pi^{k_0}, P\in\mathcal{P}^{k_0}}}\Pt(\A, \hz(\A))}{\sup\limits_{{\pi \in  \Pi^{k'}, P\in\mathcal{P}^{k'}}}\Pt(\A, \hz^*(\A)} > 0
\end{equation}
For the estimator $\hkNML$ and $k'\leq k_0$, it is enough to show that, eventually almost surely when $n\to \infty$,
\begin{equation}
    \liminf_{n\to \infty}\frac{1}{\rho_n n^2}\log \frac{\NML_{k_0}(\A )}{\NML_{k'}(\A)} > 0.
\end{equation}
Thus, we write
\begin{equation}
\begin{split}
\frac{1}{\rho_n n^2}\log \frac{\NML_{k_0}(\A)}{\NML_{k'}(\A )} &=  \frac{1}{\rho_n n^2}\log \frac{\sup\limits_{{\pi \in  \Pi^{k_0}, P\in\mathcal{P}^{k_0}}}\Pt(\A)}{\sup\limits_{{\pi \in  \Pi^{k'}, P\in\mathcal{P}^{k'}}}\Pt(\A)} \\
& -\frac{1}{\rho_n n^2}\log\CNML(n,k_0) + \frac{1}{\rho_n n^2}\log\CNML(n,k').
\end{split}
\end{equation}
By Proposition \ref{prop:razao_L_KT} we have that
\begin{equation}
    \begin{split}
         \frac{1}{\rho_n n^2}\log \CNML(n,k') -\frac{1}{\rho_n n^2}\log \CNML(n,k_0) \geq  -\frac{1}{n^2}\CNML(n,k_0) \geq  -\frac{(k_0(k_0+2)-1)\log n-2d_{k_0}}{2\rho_n n^2}
    \end{split}
\end{equation}
goes to zero as $n\to \infty$. We consider again the case that $k'=k_0-1$. Let the pair $(\Z,\A)$ be generated from the SBM with $k_0$ communities and parameters $\pi_0$ and $P_0$. Thus, we write
\begin{equation}
\sup\limits_{{\pi \in  \Pi^{k_0}, P\in\mathcal{P}^{k_0}}}\Pt(\A) = \sup\limits_{{\pi \in  \Pi^{k_0}, P\in\mathcal{P}^{k_0}}}\sum_{\e \in [k_0]^n}\Pt(\A,\e) \geq \sup\limits_{{\pi \in  \Pi^{k_0}, P\in\mathcal{P}^{k_0}}}\Pt(\A,\Z)
\end{equation}
and
\begin{equation}
\begin{split}
\log\sup\limits_{{\pi \in  \Pi^{k_0-1}, P\in\mathcal{P}^{k_0-1}}}\Pt(\A) &=\log \sup\limits_{{\pi \in  \Pi^{k_0-1}, P\in\mathcal{P}^{k_0-1}}}\sum_{\e \in [k_0-1]^n}\Pt(\A,\e)\\
&\leq \sup\limits_{{\pi \in  \Pi^{k_0-1}, P\in\mathcal{P}^{k_0-1}}}\Pt(\A,\z^*(\A) + n\log (k_0-1).
\end{split}
\end{equation}
Thus, applying again Lemma \ref{lemma:cerqueira20} we have that
\begin{equation}
\begin{split}
 \liminf_{n\to \infty} &\frac{1}{\rho_n n^2}\log \frac{\NML_{k_0}(\A)}{\NML_{k_0-1}(\A )} \\
 &\geq  \liminf_{n\to \infty}\frac{1}{\rho_n n^2}\log \frac{\sup\limits_{{\pi \in  \Pi^{k_0}, P\in\mathcal{P}^{k_0}}}\Pt(\A,\z)}{\sup\limits_{{\pi \in  \Pi^{k_0-1}, P\in\mathcal{P}^{k_0-1}}}\Pt(\A,\z^*(\A))} - \liminf_{n\to \infty}\frac{\log (k_0-1)}{\rho_n n}  \\
& > 0.
\end{split}
\end{equation}
Finally,  the result for $\hkNMCL$ follows as the result for $\hk$ by observing that
\begin{equation}
\begin{split}
\frac{1}{\rho_n n^2}\log &\frac{\NMCL_{k_0}(\A, \hz(\A) )}{\NMCL_{k}(\A, \hz^*(\A) )} =  \frac{1}{\rho_n n^2}\log \frac{\sup\limits_{{\pi \in  \Pi^{k_0}, P\in\mathcal{P}^{k_0}}}\Pt(\A, \hz(\A))}{\sup\limits_{{\pi \in  \Pi^{k}, P\in\mathcal{P}^{k}}}\Pt(\A, \hz^*(\A))} \\
& -\frac{1}{\rho_n n^2}\log \CNMCL(\hz(\A),k_0) + \frac{1}{\rho_n n^2}\log\CNMCL(\hz^*(\A),k),
\end{split}
\end{equation}
and by the uniform bound obtained in Proposition \ref{prop:razao_L_KT} we have that
\begin{equation}
    \begin{split}
         \frac{1}{\rho_n n^2}\log\CNMCL(\hz^*(\A),k') -\frac{1}{\rho_n n^2}\log \CNMCL(\hz(\A),k_0) &\geq  -\frac{1}{\rho_n n^2}\CNMCL(\hz(\A),k_0) \\
         &\geq  -\frac{(k_0(k_0+2)-1)\log n-2c_{k_0}}{2\rho_n n^2}
    \end{split}
\end{equation}
goes to zero as $n\to \infty$.
\end{proof}

\end{document}